\documentclass[11pt]{amsart}
\usepackage{amsxtra, amsfonts, amsmath, amsthm, amstext, amssymb, amscd, mathrsfs, verbatim, color}
\usepackage{threeparttable}  
\usepackage{mathtools} 
\usepackage[ansinew]{inputenc}\usepackage[T1]{fontenc}
\usepackage[all,cmtip]{xy}
\usepackage[ngerman,french,english]{babel}

\theoremstyle{plain}

\newtheorem{theorem}{Theorem}[section]
\newtheorem{lemma}[theorem]{Lemma}
\newtheorem{definition-theorem}[theorem]{Definition-Theorem}
\newtheorem{proposition}[theorem]{Proposition}
\newtheorem{corollary}[theorem]{Corollary}

\newtheorem{conjecture}[theorem]{Conjecture}

\theoremstyle{definition}
\newtheorem{definition}[theorem]{Definition}
\newtheorem{example}[theorem]{Example}
\newtheorem{remark}[theorem]{Remark}
\newtheorem{notation}[theorem]{Notation}
\newcommand \bth[1] { \begin{theorem}\label{t#1} }
\newcommand \ble[1] { \begin{lemma}\label{l#1} }

\newcommand \bpr[1] { \begin{proposition}\label{p#1} }
\newcommand \bco[1] { \begin{corollary}\label{c#1} }
\newcommand \bde[1] { \begin{definition}\label{d#1}\rm }
\newcommand \bex[1] { \begin{example}\label{e#1}\rm }
\newcommand \bre[1] { \begin{remark}\label{r#1}\rm }

\newcommand \bnota[1] {\begin{notation}\label{n#1}\rm }
\newcommand {\ele} { \end{lemma} }

\newcommand {\epr} { \end{proposition} }
\newcommand {\eco} { \end{corollary} }
\newcommand {\ede} { \end{definition} }
\newcommand {\eex} { \end{example} }
\newcommand {\ere} { \end{remark} }
\newcommand {\enota} { \end{notation} }

\begin{document}
\setlength{\baselineskip}{1.2\baselineskip}
\title[Casselman-Jacquet modules]
{Stability of the smooth Casselman-Jacquet functor}
\author[Kei Yuen Chan]{Kei Yuen Chan}
\author[Kaidi Wu]{Kaidi Wu}
\author[Jun Yu]{Jun Yu}
\author[Hongfeng Zhang]{Hongfeng Zhang}

\address{(Chan) Department of Mathematics, The University of Hong Kong, HK.}\email{kychan1@hku.hk}

\address{(Wu) Department of Mathematics and New Cornerstone Science Laboratory, The University of Hong Kong, HK.}
\email{kaidiwu24@connect.hku.hk}

\address{(Yu) School of Mathematical Sciences and Beijing International Center for Mathematical Research, Peking, 100871, China}
\email{junyu@bicmr.pku.edu.cn}

\address{(Zhang) School of Mathematics and Statistics, Huazhong University of Science and Technology, Wuhan, 430074, China}
\email{zhanghongf@pku.edu.cn}

\maketitle

\begin{abstract}
We establish and prove several results for the smooth Casselman-Jacquet submodule and quotient functors for real reductive groups. In addition to exactness, surjectivity, transitivity and globalization results, we establish a stability property for the intersection of Jacquet subspaces. Our approach is based on a family of seminorms on the Casselman-Jacquet quotient module. As an application, we establish a full version of the real Bernstein-Zelevinsky filtrations for smooth Fr\'echet representations of moderate growth.
\end{abstract}

\selectlanguage{english} 


\section{Introduction}

\subsection{Motivations}
\noindent
The Casselman--Jacquet functor is a systematic tool for studying the restriction of representations of a reductive group to its parabolic subgroups. It was developed in the foundational work of Casselman, building on Harish-Chandra's suggestive parallels between the real and $p$-adic worlds---an idea often referred to as the Lefschetz principle \cite{Ca78}. However, the significant topological differences between real and $p$-adic groups introduce extra difficulties when studying the Casselman--Jacquet functor in the real setting.

A main goal of this article is to establish several fundamental properties of the smooth Casselman--Jacquet functor. Classical treatments of this functor have focused predominantly on the algebraic side, namely Harish-Chandra modules. For instance, an application appears in the Casselman embedding theorem \cite{CM82, SW82}; computations on principal series representations can be found in \cite{Ab08};  and various functorial aspects are studied in \cite{CGYD19}. Nevertheless, Casselman-Jacquet modules have also been investigated on the analytic side---i.e., for smooth representations---in earlier work \cite{AGS15a, AGS15b}.

The smooth Casselman--Jacquet functor is particularly important in problems such as distinction, where non-admissible representations naturally arise. In such settings, the algebraic and analytic perspectives do not always perfectly coincide. A central contribution of this article is to overcome several topological obstacles in establishing these desired properties.

Our methodology on the study of Casselman-Jacquet functors makes extensive use of seminorms, which are fundamental to the theory of representations of moderate growth (see \cite{Ca89, Wa92, BK14}). Indeed, the articles \cite{BK14, BGKKS25+} have extensively studied norms on Harish-Chandra modules, for example by proving the Sobolev equivalence of continuous norms on these modules.

This work is part of a broader project \cite{Ch23, CW26} to establish a comprehensive mathematical proof of branching laws for real general linear groups, following the blueprint previously laid out for the $p$-adic case \cite{Ch21, CS21, Ch23+, CP25+}. Since the Bernstein-Zelevinsky theory is essential to the entire $p$-adic framework, it is therefore desirable to develop a parallel theory for the real case. The development of the smooth Casselman-Jacquet functor is designed for this.

A main result of this article provides a resolution to the problem of establishing the real Bernstein--Zelevinsky filtration. Our results on this filtration complement those in \cite{AGS15a, AGS15b}, with a particular focus on representations of moderate growth. Beyond the Casselman--Jacquet functor, our approach relies on the analogy between Hecke algebras on the $p$-adic side (see \cite{BZ76}) and Schwartz algebras on the real side (see \cite{dCl91}). The primary insight of our study builds on the basic cases treated in \cite{Ch23}, with further modifications discussed in \cite{WZ25+}. Related problems, such as Duflo's conjecture, are studied in \cite{LOY23}.


\subsection{Main results}

Let $G$ be a reductive Lie group and let $P$ be a parabolic subgroup of $G$ with its unipotent radical $N_P$. Let $\mathfrak n=\mathrm{Lie}(N_P)$. Let $\mathrm{Rep}^{\infty, F}(G)$ (resp. $\mathrm{Rep}^{\infty,F}(P)$) be the category of smooth Fr\'echet representations of $G$ (resp. $P$) with moderate growth. A primary source of such representations of $G$ or $P$ comes from Casselman-Wallach representations \cite{Ca89, Wa92}, possibly from a larger reductive group that contains $G$. For other examples, see e.g. \cite{dCl91, BK14, CSu21}. Main examples in this article are those restricted from general linear groups to their mirabolic subgroups.

 Define two functors:
\[\mathbf{CJ}^{\infty}_P: \mathrm{Rep}^{\infty, F}(G) \rightarrow \mathrm{Rep}^{\infty, F}(P) , \quad   \mathbf{CJ}^{\infty}_P(\pi) := \varprojlim_k \pi/\overline{\mathfrak n^k.\pi} ,
\]
and 
\[  \mathbf{CJ}^{\infty}_{s, P}: \mathrm{Rep}^{\infty, F}(G) \rightarrow \mathrm{Rep}^{\infty, F}(P), \quad    \mathbf{CJ}^{\infty}_{s,P}(\pi) := \bigcap_{k =0}^{\infty} \overline{\mathfrak n^k.\pi}  .
\]
We may sometimes use $\mathbf{CJ}^{\infty}_{N_P}$ for $\mathbf{CJ}^{\infty}_P$, and use $\mathbf{CJ}^{\infty}_{s, N_P}$ for $\mathbf{CJ}^{\infty}_{s,P}$.

We emphasize that our Casselman-Jacquet functors are defined with closure. In other words, the quotients are automatically Hausdorff. The definition suffices for representation-theoretic applications as long as they are computable; and the obvious merit is that our results are unconditional.

A starting observation of the entire article is the following:

\begin{proposition} (=Corollary \ref{cor converging seq 0})
Let $\pi$ be in $\mathrm{Rep}^{\infty, F}(G)$. Let $P$ be a parabolic subgroup of $G$ and let $\mathfrak n=\mathrm{Lie}(N_P)$.  For each $k \in \mathbb Z_{\geq 1}$, let $x_k$ be in $\overline{\mathfrak n^k.\pi}/ \mathbf{CJ}_{s}^{\infty}(\pi)$. Then $(x_k)$ is a sequence converging to $0$ in 
$\pi/\mathbf{CJ}_{s}^{\infty}(\pi)$. 
\end{proposition}

This leads to the following properties:

\begin{theorem} \label{thm stable property}
Let $\pi$ be in $\mathrm{Rep}^{\infty, F}(G)$. Let $P$ be a parabolic subgroup of $G$ and let $\mathfrak n=\mathrm{Lie}(N_P)$. Then the following statements hold:
\begin{enumerate}
\item (=Corollary \ref{cor maximal stable quotient}) $\overline{\mathfrak n.\mathbf{CJ}_{s, P}^{\infty}(\pi)}=\mathbf{CJ}_{s, P}^{\infty}(\pi)$.
\item (=Proposition \ref{prop surj cj quotient}) The map from $\pi$ to $\mathbf{CJ}_P^{\infty}(\pi)$ is surjective, and hence we have the following strict short exact sequence:
\[   0 \rightarrow \mathbf{CJ}^{\infty}_{s, P}(\pi)\rightarrow \pi \rightarrow \mathbf{CJ}_P^{\infty}(\pi) \rightarrow  0 .
\]
\item (=Propositions \ref{prop monic epi property} and \ref{prop surj cj quotient functor}) For any epi morphism $f: \pi_1\rightarrow \pi_2$, the maps $\mathbf{CJ}_P^{\infty}(f)$ and $\mathbf{CJ}_{s, P}^{\infty}(f)$ are also epi. For any surjective morphism from $\pi_1$ to $\pi_2$, $\mathbf{CJ}^{\infty}_P(f)$ is also surjective.
\item (=Proposition \ref{prop trans CJ}) Let $Q$ be a parabolic subgroup of $G$ with $Q\subset P$. Then $\mathbf{CJ}^{\infty}_Q\circ \mathbf{CJ}^{\infty}_P(\pi) \cong \mathbf{CJ}^{\infty}_Q(\pi)$.
\end{enumerate}
\end{theorem}

\begin{remark}
The map from $\pi$ to $\mathbf{CJ}^{\infty}_P(\pi)$ behaves quite different between the $(\mathfrak g, K)$-modules and representations in $\mathrm{Rep}^{\infty,F}(G)$. In the $(\mathfrak g, K)$-module case, the corresponding map is injective \cite[Corollary 8.19]{CM82}, rather than surjective. This difference arises from some analytic behaviors due to completeness. 
\end{remark}

\begin{remark}
We shall refer to Theorem \ref{thm stable property}(1) for the stability of the Casselman-Jacquet submodule. This property is useful for describing the Casselman-Jacquet submodule functor in terms of Schwartz algebra actions (see Corollary \ref{cor closedness and stable prop}).
\end{remark}

\begin{remark}
In the theory of finitely generated modules over a Noetherian commutative ring, a proof of the stabilization analogue of Theorem \ref{thm stable property}(1) relies on the Artin-Rees lemma (see, e.g., the proof of \cite[Corollary 5.4]{Ei95}). However, we currently lack a version of the Artin-Rees lemma for smooth  Fr\'echet representations of moderate growth (cf. Lemma \ref{lem AR lemma}). 

It is also known that stabilization fails for non-finitely generated modules (without a topology) over a Noetherian commutative ring.
\end{remark}

Our next result is the exactness of the Casselman-Jacquet functors for the category of Casselman-Wallach representations: 

\begin{theorem} (= Theorem \ref{thm exact cj functors cw} + Corollary \ref{cor exact cw cj sub}) \label{thm intro exact}
Let $\mathcal{CW}(G)$ be the category of Casselman-Wallach representations of $G$. Let $\iota: \mathcal{CW}(G)\rightarrow \mathrm{Rep}^{\infty,F}(G)$ be the embedding of $\mathcal{CW}(G)$ into $\mathrm{Rep}^{\infty, F}(G)$. Then $\mathbf{CJ}_{s,P}^{\infty} \circ \iota$ and $\mathbf{CJ}_P^{\infty} \circ \iota$ send a short exact sequence to a strict short exact sequence.
\end{theorem}

\begin{remark}
For $p$-adic groups, the Jacquet functor can be realized as a Jacquet integral (see \cite[Lemma 2.33]{BZ76}), and this is important for establishing its exactness. It is well-known that the Jacquet functor for real groups, in the sense of Definition \ref{def jacquet functor}, is not exact. 
\end{remark}

The exactness of the Casselman-Jacquet functor $\mathbf{CJ}^{\infty}$ cannot be deduced immediately from the Harish-Chandra setting of \cite{Ca78}, since the relevant functors may map representations outside the class of admissible ones. To circumvent this difficulty and prove exactness, we rely on a suitable variant of the Artin-Rees lemma (see Lemma \ref{lem AR lemma}).

In contrast, in terms of computational aspects, it is more possible to transport from some known results of Casselman--Jacquet quotients of Harish-Chandra modules, and so this adds value to the study of those smooth functors. A transfer result in this direction is the following, and we direct the reader to the body of the article for unexplained terminologies: 

\begin{theorem} (=Theorem \ref{prop cw jacquet functor}) \label{thm transfer thm}
Let $\pi$ be a Casselman--Wallach representation of $G$, and let $\pi_K$ be the associated Harish-Chandra module of $\pi$ i.e. the subspace of $K$-finite vectors in $\pi$, where $K$ is a maximal compact subgroup of $G$. Then the partial Casselman--Jacquet functor (in the sense of Definition \ref{def partial cj cw}) of $\pi$ is the Casselman--Wallach globalization of the corresponding partial Casselman--Jacquet functor  (in the sense of Definition \ref{def partial cj hc}) of $\pi_K$. 
\end{theorem}



\subsection{Application to real BZ filtrations}
The $p$-adic Bernstein-Zelevinsky (BZ) theory is developed in \cite{BZ76, BZ77} with application to the classification of irreducible smooth representations. One sees e.g. \cite{Ch23+} for some recent developments on the subject. Roughly speaking, the BZ theory is to study a smooth $\mathrm{GL}_n$-representation restricted to its mirabolic subgroup. 

To incorporate the smooth Casselman-Jacquet theory in our study of a BZ filtration, we need more notation. Let $\mathbb K\in \{\mathbb R, \mathbb C\}$ and let:
\[  G_n=\mathrm{GL}_n(\mathbb K), \quad V=V_{n-1}=\left\{ \begin{pmatrix} I_{n-1} & v \\ & 1 \end{pmatrix} : v \in \mathbb K^{n-1} \right\}, \quad \mathfrak{v}_{n-1}:=\mathrm{Lie}(V_{n-1}).
\]
Let $M_n$ be the mirabolic subgroup of $G_n$. Let $\mathfrak v_{n-1}^*$ be the linear dual of $\mathfrak v_{n-1}$. According to Theorem \ref{thm stable property}(2), 
one has a short exact sequence: for $\pi$ in $\mathrm{Rep}^{\infty, F}(M_n)$,
\[   0 \rightarrow \mathbf{CJ}^{\infty}_{s,V}(\pi) \rightarrow \pi \rightarrow \mathbf{CJ}_V^{\infty}(\pi) \rightarrow 0.
\]

A key issue in the BZ filtration is to provide an alternate description of $\mathbf{CJ}_{s,V}^{\infty}(\pi)$, which is Proposition \ref{prop bz description} below. The geometric origin of the description is based on the two $M_n$-orbits in $\mathfrak v_{n-1}^*$:
\[   \mathfrak v_{n-1}^*\setminus \left\{ 0 \right\} \hookrightarrow \mathfrak v_{n-1}^* \hookleftarrow \left\{ 0 \right\} .
\]
One may view that $\mathbf{CJ}^{\infty}_{s,V}(\pi)$ comes from the open orbit  $\mathfrak v_{n-1}^*\setminus \left\{ 0 \right\}$ while $\mathbf{CJ}_V^{\infty}(\pi)$ comes from the closed orbit $\left\{ 0 \right\}$. Proposition \ref{prop bz description}(1) below shall make the meaning of the sentence clearer.

We now need more ingredients to complete our story. Let $\phi$ be a twisted character on $V_{n-1}$ defined as in Section \ref{label character twist}. Define $\Phi^-:\mathrm{Rep}^{\infty, F}(M_n)\rightarrow \mathrm{Rep}^{\infty, F}(M_{n-1})$ given by
\[   \Phi^-(\pi) :=\delta^{-1/2}\cdot \pi/\mathfrak{v}_{n-1}.(\pi \otimes \phi^{-1}) ,
\]
where $\delta$ is a normalized character of $M_{n-1}$, and $\phi$ is a non-degenerate unitary character of $V_{n-1}$. The well-definedness of $\Phi^-$ is Proposition \ref{prop bz description}(2). It may be remarkable to point out that $\Phi^-$ admits a natural right adjoint functor via the Frobenius reciprocity, while admitting a left adjoint functor is an interesting property in BZ theory. We shall denote the left adjoint functor of $\Phi^-$ by $\Phi^+$, which is in the form of a Schwartz induction (see Section \ref{def_phi+} for details). 

The moderate growth condition allows one to define an action of the Schwartz algebra $\mathcal S(\mathfrak{v}_{n-1})$ on $\pi$ i.e. there exists a continuous map:
\[   a: \mathcal S(\mathfrak{v}_{n-1})\widehat{\otimes} \pi \rightarrow \pi .
\]

    Let $\mathcal S(\mathfrak{v}_{n-1}^*\setminus \left\{ 0 \right\})$ be the subspace of all functions in $\mathcal S(\mathfrak{v}_{n-1}^*)$ with all derivatives vanishing at $0$. Let $\mathcal C_k(\mathfrak v_{n-1})$ be the space spanned by the $k$-th derivatives of functions in $\mathcal S(\mathfrak v_{n-1})$ (see (\ref{eqn k derivative space})). Under the Fourier transform in Section \ref{ss fourier transform}, it takes $\mathcal S(\mathfrak{v}_{n-1}^*\setminus \left\{ 0 \right\})$ to $\bigcap_{k=1}^{\infty}\mathcal C_k(\mathfrak v_{n-1})$. The space $\mathcal S(\mathfrak{v}_{n-1}^*\setminus \left\{0\right\}).\pi$ is defined to be the image of $(\bigcap_{k=1}^{\infty}\mathcal C_k(\mathfrak v_{n-1}))\widehat{\otimes} \pi$ under the map $a$ above. 
    
    The role of the Fourier transform is to help transform the space $\mathcal S(\mathfrak{v}_{n-1}^*\setminus \left\{ 0 \right\}).\pi$ into the mirabolic induction in Section \ref{s descrip bz layer}. This will then allow one to apply the imprimitivity theorem in Section \ref{s imprimitivity thm} to obtain several desired properties, summarized as follows:

\begin{proposition} \label{prop bz description}
Let $\pi$ be in $\mathrm{Rep}^{\infty, F}(M_n)$. Then the following statements hold:
\begin{enumerate}
\item (cf. Corollary \ref{cor closedness and stable prop}) $\mathbf{CJ}_{s,V}^{\infty}(\pi)=\overline{\mathcal S(\mathfrak{v}_{n-1}^*\setminus \left\{ 0 \right\}).\pi}$.
\item (=Lemma \ref{lem kernel schwartz}(2)) $\Phi^-(\pi)$ is Hausdorff.
\item (=Corollary \ref{cor hausdorff space}) There exists a continuous bijective $M_n$-equivariant map  from $\Phi^+\circ \Phi^-(\pi)$ to $\mathcal S(\mathfrak{v}_{n-1}^*\setminus \left\{ 0 \right\}).\pi$.
\end{enumerate}
\end{proposition}

The upshot of Proposition \ref{prop bz description} is that it enables one to obtain the real BZ filtration in Theorem \ref{thm bz filtration full} through repeated applications; see Definition \ref{def bz filtration} for further details.

To understand the successive subquotients of the Bernstein-Zelevinsky filtrations, the problem naturally reduces to two tasks: (i) studying the module $\mathbf{CJ}^{\infty}_{V_{n-k}}\circ (\Phi^-)^{k-1}(\pi)$ , and (ii) describing the closure of these modules in the sense of Proposition \ref{prop bz description}. With regard to these, we further offer two additional remarks.

\begin{remark}
For (i) above, there are much progress for $p$-adic group case (see \cite{Ch23+} and references therein). For real group case, we expect that certain structural features (e.g. simple quotients) of the associated Casselman-Jacquet functors can be computed via the tools developed in \cite{CW26, CP25+}, in addition to Theorem \ref{thm transfer thm}. Moreover, a strengthened result foreshadowed in \cite[Section 12.2]{CW26} is anticipated in the forthcoming work \cite{CW26+}; this development should further facilitate the computation of higher structures of Casselman-Jacquet functors using tools from the theory of $p$-adic groups.
\end{remark}

\begin{remark}
For (ii) , the space $\mathcal S(\mathfrak v_{n-1}^*\setminus \left\{ 0 \right\}).\pi$ is not necessarily complete and so the bijection in Proposition \ref{prop bz description}(3) is not necessarily a homeomorphism. One sees a nuclearity conjecture and discussions in Section \ref{ss continuous map homeomorphism}.
\end{remark}

\subsection{Guiding examples} \label{ss   guiding example}

We finally provide examples to illustrate the above results and guide towards proofs of the results. 

\begin{example} (Finite-dimensional representations) \label{ex finite dim}
Let $G$ be a reductive Lie group, let $P$ be any parabolic subgroup of $G$ and let $\mathfrak n=\mathrm{Lie}(N_P)$. Let $\pi$ be a finite-dimensional representation of $G$. As $\mathfrak n^k$ acts trivially on $\pi$ for sufficiently large $k$,
\[  \mathbf{CJ}^{\infty}_P(\pi)=\pi, \quad \mathbf{CJ}^{\infty}_{s,P}(\pi)=0 .
\]
\end{example}

\begin{example} (Schwartz functions on $\mathbb R^{\times}$) \label{ex schwartz function M2}
Let $M_2$ be the mirabolic subgroup of $\mathrm{GL}_2(\mathbb R)$ i.e. containing all matrices of the form $\begin{pmatrix}  t& a \\ 0 &1 \end{pmatrix}$ with $t \in \mathbb R^{\times}, a\in \mathbb R$. Let $\pi=\mathcal S(\mathbb R^{\times})$ be the space of Schwartz functions from $\mathbb R^{\times}$ to $\mathbb C$, equipped with the standard topology. See the beginning of Section \ref{ss topology on power series} for details.

 We equip $\pi$ with the $M_2$-action given by:
\[  \left( \begin{pmatrix} 1 & a \\ 0 & 1 \end{pmatrix}.f \right)(x)=e^{-2\pi\sqrt{-1}xa} f(x), \quad  \left( \begin{pmatrix} t & \\ & 1     \end{pmatrix}.f\right)(x)=f(tx) .
\]
For $N=\left\{ \begin{pmatrix} 1 & a \\ & 1 \end{pmatrix}: a \in \mathbb R\right\}$ and $\mathfrak n=\mathrm{Lie}(N)$, $\pi$ satisfies the stability property in Theorem \ref{thm stable property}(1). In other words, $\mathbf{CJ}^{\infty}_N(\pi)=0$ and $\mathbf{CJ}^{\infty}_{s,N}(\pi)=\pi$.
\end{example}

\begin{example} (Principal series for $\mathrm{GL}_2$) \label{ex gl2 ps}
Let $G=\mathrm{GL}_2(\mathbb R)$ and let $B$ be the subgroup of upper triangular matrices in $G$. For $s \in \mathbb C$, let $\chi_s: \mathbb R^{\times}\rightarrow \mathbb C^{\times}$ given by $\chi(t)=|t|^s$. We consider the following normalized parabolically induced module: for $s_1, s_2 \in \mathbb C$
\[   \mathrm{Ps}(s_1, s_2) := \mathrm{Ind}_B^{\mathrm{GL}_2(\mathbb R)}(\chi_{s_1} \boxtimes \chi_{s_2}) .
\]
Then $\mathbf{CJ}_B^{\infty}(\mathrm{Ps}(s_1, s_2))$ admits a decreasing Hausdorff filtration 
\[   \ldots \subset  F_2 \subset F_1 \subset F_0 =\mathbf{CJ}_B^{\infty}(\mathrm{Ps}(s_1, s_2)) 
\]
such that the successive subquotients $F_{\ell}/F_{\ell+1}$ are $2$-dimensional $M_2$-representation such that 
\begin{enumerate}
\item the unipotent subgroup of $B$ acts trivially; and
\item the subgroup $\left\{ \begin{pmatrix} t & \\ & 1 \end{pmatrix}: t \in \mathbb R^{\times}\right\}$ acts by the eigenvalues $|t|^{s_1+\frac{1}{2}}t^{\ell}$ and $|t|^{s_2+\frac{1}{2}}t^{\ell}$; and
\item $\mathbf{CJ}_B^{\infty}(\mathrm{Ps}(s_1,s_2)) \cong \varprojlim_k F_0/F_k$.
\end{enumerate}
Moreover, $\mathbf{CJ}_{s, B}^{\infty}(\pi) \cong \mathcal S(\mathbb R^{\times})$, where $\mathcal S(\mathbb R^{\times})$ is described in Example \ref{ex schwartz function M2}. Detailed computations can be found in \cite[Section 5]{Ch23}.  
\end{example}

\begin{example} (Discrete series for $\mathrm{GL}_2$)
We use notations in Example \ref{ex schwartz function M2}. Assume $s_1=\frac{2m-1}{2}$ and $s_2=\frac{-2m+1}{2}$ for some $m \in \mathbb Z_{\geq 1}$. Then $\mathrm{Ps}(s_1, s_2)$ is reducible with two irreducible composition factors. The unique irreducible submodule of $\mathrm{Ps}(s_1, s_2)$ is an essentially discrete series, denoted by $DS_{2m}$. By Theorem \ref{thm intro exact} and Examples \ref{ex finite dim} and \ref{ex gl2 ps}, one may deduce that $\mathbf{CJ}^{\infty}_{s, B}(DS_{2m}) \cong \mathcal{S}(\mathbb R^{\times})$, and $\mathbf{CJ}^{\infty}_B(DS_{2m})$ admits a decreasing Hausdorff filtration $\ldots \subset F_2\subset F_1\subset F_0$ whose successive subquotients $F_{\ell}/F_{\ell+1}$ satisfy that the subgroup $\left\{ \begin{pmatrix} t & \\  & 1 \end{pmatrix}: t \in \mathbb R^{\times}\right\}$ acts by the eigenvalues $|t|^{m+\ell}$ and $\mathrm{sgn}(t)|t|^{m+\ell}$. See \cite[Section 5]{Ch23} for more details.
\end{example}

\subsection{Relations of key results}

\begin{itemize}
\item SmnCJ: Property of seminorms on Casselman--Jacquet quotients (Proposition \ref{prop convergence of nilpotent action})
\item StabCJs: Stability of Casselman--Jacquet submodules (Corollary \ref{cor maximal stable quotient})
\item Surj: Surjectivity from $\pi$ to $\mathbf{CJ}^{\infty}(\pi)$ (Proposition \ref{prop surj cj quotient})
\item EpCJ: Preserving epiness/surjectivity for Casselman-Jacquet submodule and quotient functors (Propositions \ref{prop monic epi property} and \ref{prop surj cj quotient functor})
\item TrnCJ: Transitivity of Casselman-Jacquet quotient functors (Proposition \ref{prop trans CJ})
\item  CJSw : Description of Casselman-Jacquet submodules in terms of Schwartz algebra actions (Corollary \ref{cor closedness and stable prop})
\item AbCatP: The abelian categorical property for the Casselman-Wallach category of a parabolic subgroup (Corollary \ref{cor cw abelian category})
\item GbCJ: Globalization of partial Casselman-Jacquet quotient functors (Theorem \ref{prop cw jacquet functor})
\item ARL: Artin-Rees lemma for the Casselman-Wallach category (Lemma \ref{lem AR lemma})
\item ExtCJ: Exactness of Casselman-Jacquet functors in the Casselman-Wallach category (Theorem \ref{thm exact cj functors cw})
\item HsfDer: Hausdorffness of the functor $\Phi^-$ (Lemma \ref{lem kernel schwartz}(2))
\item RlBZ: Real Bernstein-Zelevinsky filtrations (Theorem \ref{thm bz filtration full})
\end{itemize}

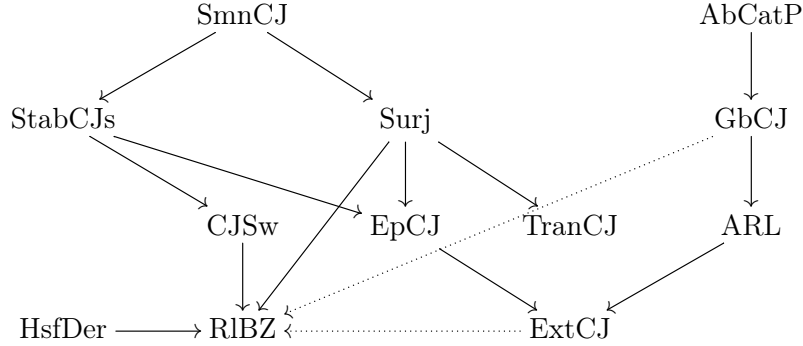
\[ \xymatrix{    & \mbox{SmnCJ} \ar[rd] \ar[ld] &      & &   \mbox{AbCatP} \ar[d] \\
 \mbox{StabCJs} \ar[rd] \ar[drr] &   &  \mbox{Surj}  \ar[d] \ar[rd] \ar[ldd] & & \mbox{GbCJ} \ar[d] \ar@{..>}[ddlll] \\
   & \mbox{CJSw} \ar[d] & \mbox{EpCJ} \ar[dr] &  \mbox{TranCJ}  &  \mbox{ARL} \ar[ld] \\
  \mbox{HsfDer} \ar[r]      &       \mbox{RlBZ}  &  &   \ar@{..>}[ll]   \mbox{ExtCJ}     &                \\
}
\]

The solid arrows above mean logical implications in our development, and the two dotted arrows mean tools for computational aspects of BZ filtrations. More discussion of the proof strategy or ideas for the proof of each individual result can be found in the corresponding sections of the results.

\subsection{Organization of the article}
Part \ref{p cj moderate growth} is devoted to the properties of Casselman-Jacquet functors. Section \ref{s prelim} sets up basic notations, and particularly reviews the quotient topology and Fourier transform. Section \ref{s schwartz alg actions} studies the Schwartz algebra action from a unipotent subgroup. Section \ref{s seminorms cj modules} discusses a seminorm on Casselman-Jacquet modules, which is used to show a stability result in Section \ref{s n stable}, and a surjectivity result in Section \ref{s exact of cj}. Then we deduce some functorial properties in Section \ref{s epi cj sub functor}, and a transitivity result in Section \ref{s transitive cj}, and a Schwartz algebra description in Section \ref{s description intersect cj}.  

Part \ref{part cj casselman wallach rep} specializes the functors to the category of Casselman-Wallach representations. In Sections \ref{s prelim gK mod} and \ref{s rep para}, we define some categories for representations of parabolic subgroups and study their globalization.  Section \ref{adj para and jacq} studies some adjointness property of Jacquet functors, and then to establish a globalization result in Section \ref{s globalization result}. Section \ref{s artin rees lemma}, we prove an Artin-Rees lemma and an exactness result. 

Part \ref{part bz filtration} studies a full version of real Bernstein-Zelevinsky filtrations. Section \ref{s mirabolic subgp} defines the mirabolic induction and Section \ref{s imprimitivity thm} shows the imprimitivity theorem. Sections \ref{s descrip bz layer}, \ref{s closed bz layer} and \ref{s real bz} establish real Bernstein-Zelevinsky filtrations and related properties. 

\subsection{Acknowledgements}
The first-named author would like to thank Dragan Mili\v{c}i\'c and Gordan Savin for fruitful discussions during his visit at the University of Utah in March 2026. He would also like to thank Dmitry Gourevitch, Henryk Hecht, Peter Trapa, Daniel Wong and Chengbo Zhu for helpful conversations or discussions during different stages of this project. The first-named author and the second-named author are grateful to Binyong Sun for inspiring discussions. The second-named author thanks Xuhua He for his continuous support and encouragement.

 This project is supported in part by the Research Grants Council of the Hong Kong Special Administrative Region, China (Project No: 17305223, 17308324) and the National Natural Science Foundation of China (Project No. 12322120). The second-named author is partially supported by the New Cornerstone Science Foundation through the New Cornerstone Investigator Program awarded to Professor Xuhua He. He is also supported by the National Natural Science Foundation of China (Grant No. 123B1004).

\part{Basic properties of smooth Casselman-Jacquet functors} \label{p cj moderate growth}

\section{Preliminaries} \label{s prelim}

\subsection{General notions}

Let $G$ be a Lie group. A representation $\pi$ of $G$ is said to be {\it continuous} if $\pi$ is a topological vector space over $\mathbb C$ such that the map
\[   G \times \pi \rightarrow \pi
\]
is continuous.

Let $\pi$ be a continuous representation of $G$. A vector $v \in \pi$ is said to be smooth if the map $G \to \pi$ given by $g \mapsto \pi(g)v=g.v$ is a smooth map. We say that $\pi$ is {\it smooth} if each vector in $\pi$ is smooth. We say that $\pi$ is {\it Fr\'echet} if the underlying topology is Fr\'echet, i.e. the topology is defined by a metric and the space is Hausdorff, locally convex and complete under the topology.

For a reductive Lie group $G$ (see \cite[Chapter 0, Sec. 3.1]{BW00}), we fix an algebraic scale $||\cdot ||_G$ on $G$ as in \cite[2.1.2]{BK14} (also see \cite[Section 2.A.2]{Wa88}). A smooth representation $\pi$ of $G$ is of {\it moderate growth} if for every continuous seminorm $p_1$ of $\pi$, there exists an integer $r \in \mathbb Z_{\geq 0}$ and a continuous seminorm $p_2$ of $\pi$ such that, for all $v \in \pi$ and all $g \in G$,
\[   p_1(\pi(g) v) \leq  ||g||^r_G \cdot p_2(v) .
\]

Let $\mathrm{Rep}^{\infty, F}(G)$ be the category of smooth Fr\'echet representations of $G$ of moderate growth,  whose morphisms are continuous $G$-equivariant linear maps. A representation $\pi$ in $\mathrm{Rep}^{\infty, F}(G)$ is said to be (topologically) irreducible if there is no proper non-zero closed $G$-invariant subspace of $\pi$. We remark that the category $\mathrm{Rep}^{\infty, F}(G)$ is not abelian, but it is quasi-abelian, see, for example, \cite[Section 2]{Ka93} for more discussions. Nevertheless, for $\pi$ in $\mathrm{Rep}^{\infty, F}(G)$, any closed $G$-invariant subspace of $\pi$ is still in $\mathrm{Rep}^{\infty, F}(G)$ and any Hausdorff $G$-invariant quotient of $\pi$ is also in $\mathrm{Rep}^{\infty, F}(G)$, see \cite[Lemma 11.5.2]{Wa92} or \cite[Lemma 2.9]{BK14}.

For a parabolic subgroup $P$ of $G$, write $N_P$ to be the unipotent radical of $P$ and $M_P$ to be the Levi subgroup of $P$.

\subsection{Quotient topology} \label{ss quotient topology}

Let $\pi$ be a Fr\'echet space. Let $\left\{ p_i \right\}_{i \in \mathbb Z_{\geq 1}}$ be a countable family of seminorms defining the topology on a vector space $\pi$. One has a metric on $\pi$ defining the topology of $\pi$ given by:
\[ d(v,v') = \sum_{i=1}^{\infty}\frac{1}{2^i}\frac{p_i(v-v')}{1+p_i(v-v')}.
\]
The Hausdorffness for $\pi$ is equivalent to:
\[   p_i(v)= 0 \mbox{ for all $i \in \mathbb Z_{\geq 1}$} \quad \Longleftrightarrow \quad v=0 .
\]
By summing or maximizing the seminorms if necessary, we may always assume the seminorms are increasing, i.e. $p_1 \leq p_2 \leq p_3 \leq \ldots$.

Now let $\kappa$ be a closed subspace of $\pi$, equipped with the subspace topology. Define a metric on $\pi/\kappa$ by 
\[ \widetilde{d}([v],[v']):=\mathrm{inf}\left\{ d(\widetilde{v}, \widetilde{v}'): \widetilde{v}\in [v], \widetilde{v}' \in [v'] \right\}.\] 
It is a classical result that the topology determined by the metric $\widetilde{d}$ agrees with the quotient topology.

Alternatively, one can define a family of seminorms as: 
\[\widetilde{p}_n([v]) :=\mathrm{inf}\left\{ p_n(\widetilde{v}) : \widetilde{v} \in [v]\right\}. \]
This family $\left\{ \widetilde{p}_n \right\}$ of seminorms defines the topology on $\pi/\kappa$. This defines another metric
\[  \overline{d}([v],[v']): = \sum_{i=1}^{\infty} \frac{1}{2^i} \frac{\widetilde{p}_i([v]-[v'])}{1+\widetilde{p}_i([v]-[v'])} .
\]
One observes that, for any $[v], [v']$ in $\pi/\kappa$, $\overline{d}([v], [v'])\leq \widetilde{d}([v],[v'])$. Hence, the topology defined by $\overline{d}$ is a priori weaker than the topology defined by $\widetilde{d}$. By the open mapping theorem, the two topologies indeed agree.

\subsection{Projective tensor products} \label{ss projective norm}

Let $\pi_1$ and $\pi_2$ be two Fr\'echet spaces. Let $\left\{ p_i\right\}_{i\in \mathbb Z_{\geq 1}}$ and $\left\{ q_i \right\}_{i \in \mathbb Z_{\geq 1}}$ be families of seminorms defining the topology of $\pi_1$ and $\pi_2$ respectively. Define the projective seminorms $\left\{ p_i\otimes q_j\right\}_{i,j\in \mathbb Z_{\geq 1}}$ as: for any $v \in \pi_1\otimes \pi_2$ as (\cite[Chapter 43]{Tr06}):
\[  (p_i\otimes q_j)(v)= \mathrm{inf} \left\{ r>0: v \in r\omega \right\}  ,
\]
where $\omega$ is the balanced convex hull of the set $\left\{ v_1 \otimes v_2 \in \pi_1\otimes \pi_2 : p_i(v_1)\leq 1, q_j(v_2)\leq 1 \right\}$. According to \cite[Proposition 43.1]{Tr06}, for any $v \in \pi_1\otimes \pi_2$,
\[   (p_i \otimes q_j)(v) =\inf \left\{ \sum_k p_i(v_1^k)\cdot q_j(v_2^k)  \right\}  ,
\]
where $v_1^k$ and $v_2^k$ run for all the finite sums such that $v=\sum_k v_1^k \otimes v_2^k$.

The family $\left\{ p_i \otimes q_j \right\}_{i,j}$ of seminorms defines a metrizable locally convex topology on the (algebraic) tensor product $\pi_1\otimes \pi_2$. The projective tensor product $\pi_1\widehat{\otimes} \pi_2$ is the completion of $\pi_1\otimes \pi_2$ under the metric.

\subsection{Fourier transform on $\mathcal S(\mathbb R^n, \pi)$} \label{ss fourier transform}

Let $\pi$ be a Fr\'echet space (over $\mathbb C$) with the associated seminorms $q_j$, $j\in \mathbb{Z}_{\geq 1}$. We shall write the variable vector $x$ in $\mathbb R^n$ as $(x_1, \ldots, x_n)$ and the corresponding partial derivatives by $\partial_1, \ldots, \partial_n$. Let $\mathcal S(\mathbb R^n, \pi)$ be the space of $\pi$-valued Schwartz functions on $\mathbb R^n$, that is, the space of smooth function $f: \mathbb R^n\rightarrow \pi$ such that for any $k, k_1,\ldots, k_n \in \mathbb Z_{\geq 0}$, and any $j \in \mathbb Z_{\geq 1}$,
\[ \mathrm{sup}_{(x_1,\ldots, x_n)\in \mathbb R^n} \left\{ (1+|x_1|^2+\ldots+|x_n|^2)^k \cdot q_j(\partial_1^{k_1}\ldots \partial_n^{k_n}f(x_1,\ldots, x_n)) \right\} <\infty .
\]
Define a family of seminorms on $\mathcal S(\mathbb R^n, \pi)$: for $k, k_1, \ldots, k_n\in \mathbb Z_{\geq 0}$, and $j \in \mathbb Z_{\geq 1}$
\[  p_{k,k_1,\ldots, k_n}^{j}(f)=\mathrm{sup}_{(x_1,\ldots, x_n)\in \mathbb R^n} \left\{(1+|x_1|^2+\ldots+|x_n|^2)^k \cdot q_j(\partial_1^{k_1}\ldots \partial_n^{k_n}f(x_1,\ldots, x_n)) \right\}.
\]
We shall equip $\mathcal S(\mathbb R^n, \pi)$ with the topology induced by the family of seminorms $ \left\{ p_{k,k_1,\ldots, k_n}^j \right\}$.

Let $(\mathbb R^n)^*$ be the linear dual of $\mathbb R^n$. Let $\langle \cdot,\cdot \rangle$ be the non-degenerate pairing of $ (\mathbb R^n)^* \times \mathbb R^n  \rightarrow \mathbb R$ given by $\langle y,x \rangle:= y(x)$. The pairing $\langle \cdot , \cdot \rangle$ identifies $(\mathbb R^n)^*$ with $\mathbb R^n$, and let $(y_1,\ldots, y_n)$ be the dual vector of $(\mathbb R^n)^*$  i.e. $\langle (y_1,\ldots, y_n), (x_1,\ldots, x_n)\rangle=x_1y_1+\ldots +x_ny_n$. 

The Fourier transform defines a homeomorphism from $\mathcal S(\mathbb R^n, \pi)$ onto $\mathcal S((\mathbb R^n)^*, \pi)$: for $f \in \mathcal S(\mathbb R^n, \pi)$ and for $y \in (\mathbb R^n)^*$,
\[   \widehat{f}(y) = \int_{\mathbb R^n}   e^{2\sqrt{-1}\pi\langle y, x\rangle} f(x)~ dx ,
\]
with the standard Euclidean measure $dx$ and so 
\[  \widehat{\partial_{i} f}(y)=-2\sqrt{-1}\pi y_i \widehat{f}(y) , \mbox{ and }  \quad   \widehat{x_i\cdot f}(y)=  \frac{1}{2\sqrt{-1}\pi}\partial_{y_i}(\widehat{f})(y) .
\]
 Suppose further that there exists a linear group action of a Lie group $G$ on $\mathbb R^n$. This induces a continuous group action $G$ on $\mathcal S(\mathbb R^n, \pi)$ given by $(g.f)(x)=f(g^{-1}.x)$. Under the Fourier transform, for $f \in \mathcal S((\mathbb R^n)^*, \pi)$ and $g \in G$,
\[   \widehat{g.f}(y) = |\mathrm{det}(g)|\cdot (\widehat{f}( g^{-1}.y)) ,
\]
where $\det(g)$ is the determinant of the action of $g$ on $\mathbb R^n$. 


Let $\mathcal S((\mathbb R^n)^*\setminus \left\{ 0 \right\}, \pi)$ be the space of smooth functions in $\mathcal S((\mathbb R^n)^*, \pi)$, all of whose derivatives vanish at $0$, equipped with the subspace topology from $\mathcal S((\mathbb R^n)^*. \pi)$. Then the Fourier transform  restricts to a topological isomorphism between the two spaces $\mathcal S((\mathbb R^n)^*\setminus \left\{ 0 \right\}, \pi)$ and  
\[\bigcap_{i=0}^{\infty} \mathrm{Span}_{\mathbb{C}}\left\{ \partial_n^{m_n}\ldots\partial_1^{m_1}f   : f\in  \mathcal S(\mathbb R^n, \pi),  m_1, \ldots, m_n \in \mathbb Z_{\geq 0}\ \text{s.t.}\ \sum_{j=1}^{n}m_j=i \right\}. \]




\subsection{Power series $\mathbb C[[x_1, \ldots, x_n]]$} \label{ss topology on power series}


Let $\mathbb C[[x_1, \ldots, x_n]]$ be the space of formal power series of $n$-variables at $0$. Define a countable family $\left\{ \widetilde{p}_{k_1, \ldots, k_n} \right\}_{k_1,\ldots, k_n \in \mathbb Z_{\geq 0}}$ of seminorms on $\mathbb C[[x_1, \ldots, x_n]]$ given by:
\[   \widetilde{p}_{k_1, \ldots, k_n}(Q) := | (\partial_1^{k_1}\ldots \partial_n^{k_n}Q)(0) | .
\]
This family determines the topology on $\mathbb{C}[[x_1,\ldots, x_n]]$ and turns it into a Fr\'echet space. 

As topological vector spaces, the classical Borel's lemma asserts that:
\begin{align} \label{eqn borel lemma}
0 \rightarrow \mathcal S(\mathbb R^n\setminus \left\{ 0 \right\}) \stackrel{\iota}{\rightarrow} \mathcal S(\mathbb R^n) \stackrel{\mathrm{pr}}{\rightarrow} \mathbb C[[x_1, \ldots, x_n]] \rightarrow 0 ,
\end{align}
where $\iota$ is the inclusion map, and $\mathrm{pr}(f)$ is given by the Taylor series expansion of $f$ at $0$. One may see \cite[Section 2]{CHM00} for more related discussions.

Note that 
\begin{align} \label{eqn schwartz ideal 0}
x_1\mathcal S(\mathbb R^{n}\setminus \left\{ 0 \right\}) +\ldots +x_n\mathcal S(\mathbb R^{n}\setminus \left\{0 \right\}) =\mathcal S(\mathbb R^n \setminus \left\{ 0 \right\}) .
\end{align}
Indeed, for any $f \in \mathcal S(\mathbb R^n \setminus \left\{ 0 \right\})$, 
\[  f= x_1\left(\frac{x_1f}{x_1^2+\ldots +x_n^2}\right)+\ldots +x_n\left(\frac{x_nf}{x_1^2+\ldots +x_n^2}\right) \in x_1\mathcal S(\mathbb R^n \setminus \left\{ 0 \right\})+\ldots +x_n\mathcal S(\mathbb R^n\setminus \left\{ 0 \right\})  .
\]

Combining (\ref{eqn borel lemma}) and (\ref{eqn schwartz ideal 0}), one also sees that $x_1\mathcal S(\mathbb R^n)+\ldots +x_n\mathcal S(\mathbb R^n)$ is the subspace of functions in $\mathcal S(\mathbb R^n)$ vanishing at $0$.

\section{Schwartz algebra actions} \label{s schwartz alg actions}

In this section, we discuss actions from Schwartz algebras, see e.g. \cite{dCl91} for more discussions. We shall consider $G$ to be a reductive Lie group in the remaining of Part \ref{p cj moderate growth}.

\subsection{Schwartz algebra modules}

Let $N$ be the unipotent radical of a parabolic subgroup $P$ of $G$. Let $\mathfrak n=\mathrm{Lie}(N)$. As $N$ is simply-connected, the exponential map $\mathrm{exp}$ defines a diffeomorphism from $\mathfrak n$ to $N$. We shall usually identify $\mathfrak n \cong \mathbb R^n$ for some $n$, and so we have the space $\mathcal S(\mathfrak n)$ of Schwartz functions from $\mathfrak n$ to $\mathbb C$, discussed in Sections \ref{ss fourier transform} and \ref{ss topology on power series}. 

 With a fixed scale $||\cdot||_G$ on $G$ as in \cite[2.1.2]{BK14}, one can deduce that $||\mathrm{exp}(x)||_G$ ($x\in \mathfrak n$) is bounded by a polynomial function on $\mathfrak n$. Let $\pi$ be in $\mathrm{Rep}^{\infty,F}(N)$. With the moderate growth condition, we obtain a well-defined continuous map \cite[Section 11.8]{Wa92}:
\begin{align} \label{eqn n action}
a: \mathcal S(\mathfrak n) \widehat{\otimes} \pi \rightarrow \pi 
\end{align}
defined by
\begin{align} \label{eqn algebra action}
f\otimes v \mapsto \pi(f)v :=\int_{\mathfrak n} f(x) \cdot  \mathrm{exp}(x).v~ dx ,
\end{align}
where $dx$ is the Euclidean measure on $\mathfrak n$. Indeed $dx$ is the Haar measure of $N$. We shall say that the above map is the $\mathcal S(\mathfrak n)$-action on $\pi$.

It follows from the Dixmier-Malliavin Theorem that:

\begin{proposition} (see \cite[Proposition 2.20]{BK14})  \label{prop nondeg schwartz}
The $\mathcal S(\mathfrak n)$-action on $\pi$ is non-degenerate, i.e. the map in (\ref{eqn n action}) is surjective.
\end{proposition}



\subsection{Some module structure of $\mathcal S(\mathfrak n)$} \label{ss first intersect prop}

One considers $\mathcal S(\mathfrak n)$ as an $N$-representation, via the following action: for $f \in \mathcal S(\mathfrak n)$, and for $g \in N$ and $X \in \mathfrak n$,
\[  (g.f)(X) =f( \mathrm{log}(g^{-1}\mathrm{exp}(X))) .
\]
This induces an $N$-action on $\mathcal S(\mathfrak n)$, and so induces an $N$-action on $\mathcal S(\mathfrak n)\widehat{\otimes}\pi$ determined by: for $f \in \mathcal S(\mathfrak n)$ and $v \in \pi$, and for $g \in N$,
\[ g.(f\otimes v) = (g.f)\otimes v .\]
Then, with $dx$ to be a Haar measure, the map in (\ref{eqn algebra action}) is also $N$-equivariant. We can extend $\mathcal S(\mathfrak n)\widehat{\otimes} \pi$ to a $P$-representation determined by:  for $f \in \mathcal S(\mathfrak n)$ and $v \in \pi$, and for $g \in M_P$,
\[ g.(f\otimes v)=(g.f)\otimes (g.v) .
\]
The map in (\ref{eqn algebra action}) is also $P$-equivariant. The space $\mathcal S(\mathfrak n)$ is in $\mathrm{Rep}^{\infty, F}(P)$ (see \cite[Remark 2.19 and Proposition 2.20]{BK14}). 

Let $\mathrm{Diff}_k(\mathfrak n)$ be the space of homogeneous differential operators on $\mathfrak n$ of degree $k$. For a fixed $k \in \mathbb Z_{\geq 0}$, let
\begin{align} \label{eqn k derivative space}
\mathcal{C}_k(\mathfrak n)= \mathcal{C}_k:= \left\{ \sum_{j=1}^rD_j f_j :  f_j\in \mathcal S(\mathfrak n), D_j\in \mathrm{Diff}_k(\mathfrak n), r\in \mathbb{Z}_{\geq 1} \right\}  .
\end{align}

The proof of Lemma \ref{lem nk Schwartz desc} in the following is based on an inductive analysis of the multiplication of the corresponding unipotent radical. In the abelian case, the lemma could be simpler by simply applying the Fourier transform.

\begin{lemma}  \label{lem nk Schwartz desc}
We use the above notation. Then the space $\bigcap_{k =0}^{\infty} \mathcal C_k$ is invariant under the $\mathfrak n$-action.
\end{lemma}

\begin{proof}
Let $n=\mathrm{dim}~\mathfrak n$. Let $\mathfrak{n}_0=\mathfrak{n}$. For $i \geq 1$, inductively define $\mathfrak n_i=[\mathfrak n_{i-1},\mathfrak n]$. The nilpotency of $\mathfrak n$ implies that $\mathfrak n_i=0$ for $i \geq n$. Let $r^*$ be the largest non-negative integer such that $\mathfrak n_{i^*}\neq 0$. Now, we fix a basis $X_1, \ldots, X_{i_1}$ for $\mathfrak n_{r^*}$, and then extend to a basis $X_1, \ldots, X_{i_2}$ ($i_2>i_1$) for $\mathfrak n_{r^*-1}$. Repeat the same process and then obtain a basis $X_1, \ldots, X_n$ for $\mathfrak n$. We write $\partial_{X_i}$ for the corresponding partial differentiation with respect to $X_i$ on functions in $\mathcal S(\mathfrak n)$. For $f \in \mathcal S(\mathfrak n)$, we shall write $f(a_1,\ldots, a_n)$ to represent $f(a_1X_1+\ldots+a_nX_n)$.


Let $f \in \bigcap_{k=0}^{\infty} \mathcal C_k$. It follows from the Baker-Campbell-Hausdorff formula that: for $Y \in \mathfrak n$ and $f \in \mathcal S(\mathfrak n)$, and a variable $t$,
\[ (\mathrm{exp}(tY).f)(X) = f(X- tY-\frac{t}{2}[Y,X]+\frac{t^2}{12}[Y,[Y,X]]+\ldots ) .
\]
For $1\leq i \leq i_1$, taking $Y=X_i$ and taking differentiation, we have:
\begin{align*}
(X_if)(a_1, \ldots, a_n) &= -(\partial_{X_i}f)(a_1,\ldots, a_n)  .
\end{align*}
Hence, $X_if \in \bigcap_{k=0}^{\infty} \mathcal C_k$.

Note that, by the nilpotency on $\mathfrak n$ and our choice of a basis, for $i_1+1 \leq i \leq i_2$,
\[  -\frac{1}{2}[X_i, \sum_{j=1}^na_jX_j] =  \sum_{j=1}^{i_1} \ell_j(a_{i_1+1}, \ldots, a_n) X_j
\]
for a linear function $\ell_j$ on $a_{n_1+1},\ldots, a_n$. Note that other terms in the Baker-Campell-Hausdorff formula are zero. Consequently,
\begin{align} \label{eqn derivative formula}
(X_if)(a_1,\ldots, a_n) &=- (\partial_{X_i}f)(a_1,\ldots, a_n)  + \sum_{j=1}^{i_1} \ell_j(a_{i_1+1},\ldots, a_n) \cdot (\partial_{X_j}f)(a_1,\ldots, a_n) .
\end{align}
 Inductively, we see that the terms, for $j=1, \ldots, i_1$, $\partial_{X_j}f$
is in $\bigcap_{k=0}^{\infty} \mathcal C_k$. Then we also have $\ell_j\cdot \partial_{X_j}f$ is also in $\bigcap_{k=0}^{\infty}\mathcal C_k$. (To see this, one may apply the Fourier transform  on $\ell_j\cdot \partial_{X_j}f$, see Section \ref{ss fourier transform}. Since $\widehat{\ell_j \cdot \partial_{X_j}f}$ is equal to a derivative on $\widehat{\partial_{X_j}f}$, $\widehat{\ell_j\cdot\partial_{X_j}f}$ is in $\mathcal S(\mathfrak n^*\setminus \left\{ 0\right\})$, where $\mathfrak n^*$ is the linear dual of $\mathfrak n$. Then, by taking the inverse Fourier transform, $\ell_j\cdot \partial_{X_j}f$ is also in $\bigcap_{k=0}^{\infty}\mathcal C_k$.) 

Thus, with \eqref{eqn derivative formula},
\begin{align} \label{eqn derivative form double}
\partial_{X_i}f \in \bigcap_{k=0}^{\infty} \mathcal C_k. 
\end{align}

For the general case, one has to replace (\ref{eqn derivative formula}) with a more general formula:  for $i_s+1\leq i \leq i_{s+1}$,
\begin{align} \label{eqn general formula}
(X_if)(a_1, \ldots, a_n) = -(\partial_{X_i}f)(a_1,\ldots, a_n) + \sum_{t=0}^{s-1} \sum_{j=i_t+1}^{i_{t+1}} \ell_j(a_{i_{t+1}+1},\ldots, a_n) \cdot (\partial_{X_j}f)(a_1,\ldots, a_n) ,
\end{align}
where $\ell_j$ (for $i_{t}+1\leq j \leq i_{t+1}$) is a polynomial in variables $a_{i_{t+1}+1},\ldots, a_n$, and is of degree at most $r^*-1$. Thus one now proceeds inductively with similar argument above to show that other derivatives of the form (\ref{eqn derivative form double}) (but now $i$ to be arbitrary in $\left\{ 1, \ldots, n \right\}$) are in $\bigcap_{k=0}^{\infty}\mathcal C_k$.
\end{proof}

\begin{lemma} \label{lem intersect jacquet}
For each $k \in \mathbb Z_{\geq 1}$, there exists a sufficiently large positive integer $p$ such that 
\[  \mathfrak n^p.\mathcal S(\mathfrak n) \subset \mathcal C_k .
\]
In particular, $\bigcap_{k=0}^{\infty} \mathfrak n^k.\mathcal S(\mathfrak n) \subset \bigcap_{k=0}^{\infty} \mathcal C_k$.
\end{lemma}

\begin{proof}
We shall use the notations in the proof of Lemma \ref{lem nk Schwartz desc}. We pick $p=k(r^*)^{r^*}$. Let $j_1, \ldots, j_p \in \left\{ 1, \ldots, n \right\}$ and let $f \in \mathcal S(\mathfrak n)$. Our goal is to show that 
\[     X_{j_p} \ldots X_{j_1}f  \in \mathcal C_k . 
\]

Using the expression (\ref{eqn general formula}), it suffices to show that 
\[                 D_{j'_p}  \ldots  D_{j'_1}f \in \mathcal C_k ,
\]
where each $D_{j_a'}$ ($j_a'\in \left\{1, \ldots, n\right\}$) takes the form $\ell_a \cdot \partial_{X_{j'_a}}$ with certain polynomial function $\ell_a$ of degree at most $r^*-1$.

For $s=0,1, \ldots r^*-1$, let
\[  x_s = \mathrm{card} \left\{ j_c' \in  \left\{ i_s+1,\ldots, i_{s+1} \right\}: c=1,\ldots p   \right\}
\]
i.e. the number of times of $j_c$'s appearing in $\left\{ i_s+1, \ldots, i_{s+1} \right\}$.

Suppose $x_0 \geq k$. Note that from (\ref{eqn general formula}), $\ell_a$ is independent of $X_1,\ldots X_{i_1}$. Thus, we can rearrange the term $D_{j_p'}\ldots D_{j_1'}f$ into the form:
\[   \partial_{X_{i_1}}^{m_{i_1}} \ldots \partial_{X_1}^{m_1}\widetilde{f}
\]
for some $\widetilde{f}\in \mathcal S(\mathfrak n)$ and some $m_{1}, \ldots, m_{i_1}$ with $m_1+\ldots +m_{i_1} \geq k$. This implies the expression is in $\mathcal C_k$ and so we are done.

We consider the product of the polynomial functions
\[  \ell_p\cdot \ldots \cdot \ell_1 
\]
and let the sum of the degrees of $a_{i_1+1}, \ldots a_{i_2}$ of the polynomial function is $S$. Note that
\[   S <  k(r^*-1) ,
\]
because the contributions from $a_{i_1+1},\ldots, a_{i_2}$ have to come from $D_{j_s'}$ with $j_s'\in \left\{ 1,\ldots, i_1\right\}$, but the condition $x_0$ imposes a constraint on $S$.

Suppose $x_0 < k$ and $x_1 \geq kr^*$. By applying the product formula, $D_{j_p'}\ldots D_{j_1'}f$ can be expressed into the form:
\begin{align} \label{eqn linear expressions}
\sum_{m_1'+\ldots+m_n' \leq p}  \ell'_{(m_1',\ldots, m_n')} \cdot (\partial_{X_n}^{m_n'}\ldots \partial_{X_1}^{m_1'}f)
\end{align}
for some polynomial function $\ell'_{m_1',\ldots, m_n'}$ (possibly zero).

Note from the product formula that, if $\ell'_{m_1',\ldots, m_p'}$ is non-zero, then we must have that $m_{i_1+1}'+\ldots +m_{i_2}'$ is greater than or equal to $k$ plus the sum of degrees of $a_{i_1+1},\ldots,  a_{i_2}$ of $\ell'_{(m_1',\ldots, m_p')}$.

To produce the terms in (\ref{eqn linear expressions}),  one considers the expression
\begin{align} \label{eqn rearrange the term}       \sum_{m_1'+\ldots +m_n' \leq p} \partial_{X_n}^{m_n'}\ldots  \partial_{X_1}^{m_1'} (\ell'_{(m_1',\ldots, m_n')}\cdot f)  .
\end{align}
Note that the above discussion guarantees that $m_{i_1+1}'+\ldots +m_{i_2}'\geq k$, and so is in $\mathcal C_k$.

The expression (\ref{eqn rearrange the term}) after the product formula produces more terms than (\ref{eqn linear expressions}), but those terms are of smaller degree and still satisfies above discussions. Now one repeats the same consideration as above, and then eventually one can show that $D_{j_p'}\cdots D_{j_1'}f$ is in $\mathcal C_k$.

To consider more general case, we suppose 
\begin{align} \label{eqn inequalities in approx}
x_0< k, x_1< kr^*, \ldots, x_s < k(r^*)^{s},\mbox{ and } x_{s+1} \geq k(r^*)^{s+1} \end{align}
for some $s$. A similar argument could show that $D_{j_p'}\ldots D_{j_1'}f \in \mathcal C_k$. 

With our choice on $p$, the inequalities (\ref{eqn inequalities in approx}) must be satisfied for some $s$. Thus we are done.
\end{proof}

We shall use other method (to avoid technical and not so insightful computations) to discuss the reverse inclusion in Lemma \ref{lem intersect jacquet}, and one sees Corollary \ref{cor jacquet intersect for schwartz}.

\subsection{Some module structure of $\mathcal S(\mathfrak n^*)$}

Let $\mathfrak n^*$ be the linear dual of $\mathfrak n$. 

\begin{lemma}
The space $\mathcal S(\mathfrak n^*\setminus \left\{ 0 \right\})$ is invariant under the action of $\mathfrak n$.
\end{lemma}

\begin{proof}
This follows from taking the Fourier transform in Section \ref{ss fourier transform} and Lemma \ref{lem nk Schwartz desc}.
\end{proof}

Indeed, we have the following stronger invariance result (cf. Corollary \ref{cor maximal stable quotient} below):

\begin{lemma} \label{lem n-invariant under schwartz}
We have $\mathfrak n.\bigcap_{k=0}^{\infty}\mathcal C_k(\mathfrak n)=\bigcap_{k=0}^{\infty}\mathcal C_k(\mathfrak n)$, and $\mathfrak n.\mathcal S(\mathfrak n^*\setminus \left\{ 0 \right\})=\mathcal S(\mathfrak n^*\setminus \left\{ 0 \right\})$. 
\end{lemma}

\begin{proof}
The two statements are equivalent by taking the Fourier transform and we shall show the second one. 

 Indeed, from the equation (\ref{eqn general formula}), one inductively shows that for $f \in \bigcap_{k= 0}^{\infty}\mathcal C_k$, and any $j=1,\ldots n$, the derivatives of $f$ are in $\mathfrak n.\bigcap_{k=0}^{\infty}\mathcal C_k$. Fix a basis $\left\{ X_1, \ldots, X_n\right\}$ for $\mathfrak n^*$. For $f \in \mathcal S(\mathfrak n^*)$, we write $f(b_1,\ldots, b_n)$ for $f(b_1X_1+\ldots+b_nX_n)$. By taking the Fourier transform, we have: for $f \in \mathcal S(\mathfrak n^*\setminus \left\{ 0 \right\})$, and any $j=1,\ldots, n$, $b_jf \in \mathfrak n.\mathcal S(\mathfrak n^*\setminus \left\{ 0 \right\})$. Then, following from (\ref{eqn schwartz ideal 0}), $\mathfrak n.\mathcal S(\mathfrak n^*\setminus \left\{0 \right\})=\mathcal S(\mathfrak n^*\setminus \left\{ 0\right\})$.
\end{proof}

\section{seminorms on Casselman-Jacquet modules} \label{s seminorms cj modules}

 Let $P$ be a parabolic subgroup of $G$. Let $\mathfrak n=\mathrm{Lie}(N_P)$ and simply write $N$ for $N_P$. In this section, we show the topology on the Casselman-Jacquet module has some nice properties from the seminorms, that is Proposition \ref{prop convergence of nilpotent action}. We outline the main idea of the proof. Firstly, the condition of moderate growth allows one to define a continuous action of the Schwartz algebra $\mathcal S(\mathfrak n)$ on $\pi$, and so defines a continuous map $\mathcal S(\mathfrak n) \widehat{\otimes} \pi \rightarrow \pi$. With the help from the classical Dixmier-Malliavin theorem forces, we have that the action is surjective (see Lemma \ref{lem non-dgenerate quotient}). Now, by the open mapping theorem, one  deduces the topology on $\pi$ from the quotient topology and the tensor product topology of $\mathcal S(\mathfrak n)\widehat{\otimes}\pi$.

\subsection{Two functors} \label{ss two functors}
 Recall that $\mathrm{Rep}^{\infty, F}(P)$ is not abelian  because for a given morphism $f: \pi_1\rightarrow \pi_2$, the space $\pi_1/f^{-1}(0)$ is not necessarily homeomorphic to $f(\pi_1)$. Note that any monic in $\mathrm{Rep}^{\infty,F}(P)$ is injective while any epi $f:\pi_1\rightarrow \pi_2$ in $\mathrm{Rep}^{\infty,F}(P)$ has dense set-theoretic image in the codomain, i.e. $\overline{f(\pi_1)}=\pi_2$.

Define the {\it Casselman-Jacquet submodule functor} from $\mathrm{Rep}^{\infty, F}(G)$ to $\mathrm{Rep}^{\infty, F}(P)$:
\[ \mathbf{CJ}_{s,P}^{\infty}(\pi)=  \mathbf{CJ}_{s,N}^{\infty}(\pi) := \bigcap_{k=0}^{\infty} \overline{\mathfrak n^k.\pi} ,
\]
equipped with the subspace topology from $\pi$. Note that the space is invariant under the action of $P$ and so is well-defined. 

Define the {\it Casselman-Jacquet (quotient) functor} from  $\mathrm{Rep}^{\infty, F}(G)$ to $\mathrm{Rep}^{\infty, F}(P)$:
\[ \mathbf{CJ}_P^{\infty}(\pi) = \mathbf{CJ}_N^{\infty}(\pi) := \varprojlim_{k } \pi/\overline{\mathfrak n^k.\pi} . \]
Recall that the inverse limit of Fr\'echet spaces is still Fr\'echet, and so this gives the topology on $\mathbf{CJ}_P^{\infty}(\pi)$. We shall regard $\mathbf{CJ}_P^{\infty}(\pi)$ as a $P$-representation.

We occasionally omit the subscript $N$ or $P$ from the preceding notation. More generally, one may replace $G$ with other Lie groups that contain $N$ as a subgroup - for instance, $N$ itself. The functor remains well-defined in these contexts except that the modules $\mathbf{CJ}^{\infty}_{s,P}(\pi)$ and $\mathbf{CJ}^{\infty}_P(\pi)$ are equipped with different group actions, and we shall occasionally make use of such variants.

\subsection{Seminorms on the Casselman-Jacquet functor} \label{ss seminorms cj functor}






In the action map (\ref{eqn n action}), we shall express 
\[ \mathcal S(\mathfrak n^*\setminus \left\{ 0 \right\}).\pi :=a( \bigcap_{k=0}^{\infty} \mathcal C_k(\mathfrak n) \widehat{\otimes} \pi).\] This notion is more convenient if one later compares with the $\mathbb C[[\mathfrak n^*]]$-action on the Casselman-Jacquet quotients.

\begin{lemma} \label{lem inclusion}
Let $\pi$ be in $\mathrm{Rep}^{\infty, F}(N)$. Then $\mathcal S(\mathfrak n^*\setminus \left\{ 0 \right\}).\pi  \subset \mathbf{CJ}^{\infty}_s(\pi)$.
\end{lemma}

\begin{proof}
Let $a$ be the action map given in (\ref{eqn n action}). The lemma then follows from:
\[  \mathcal S(\mathfrak n^*\setminus \left\{ 0 \right\}).\pi \subset a(\left(\bigcap_{k=0}^{\infty} \mathfrak n^k.\mathcal S(\mathfrak n)\right) \widehat{\otimes}\pi) \subset \bigcap_{k=0}^{\infty} (\mathfrak n^k.\pi) \subset \mathbf{CJ}^{\infty}_s(\pi) ,\]
where the first equality follows from $\mathcal S(\mathfrak n^*\setminus \left\{0 \right\})\subset \mathfrak n^k.\mathcal S(\mathfrak n^*)$ for all $k \in \mathbb Z_{\geq 0}$ by Lemma \ref{lem n-invariant under schwartz}, and the middle inclusion follows from
\[  a\left(\bigcap_{k=0}^{\infty} \mathfrak n^k.\mathcal S(\mathfrak n) \widehat{\otimes} \pi\right) \subset a\left(\mathfrak n^k.\mathcal S(\mathfrak n) \widehat{\otimes} \pi\right)=\mathfrak n^k.a\left(\mathcal S(\mathfrak n)\widehat{\otimes} \pi \right) \quad \mbox{ for all $k \in \mathbb Z_{\geq 0}$}
\]
and $a(\mathcal S(\mathfrak n) \widehat{\otimes}\pi)=\pi$ (see Proposition \ref{prop nondeg schwartz}).
\end{proof}

\begin{lemma} \label{lem non-dgenerate quotient}
Let $\pi$ be in $\mathrm{Rep}^{\infty, F}(N)$. Then there exists a non-degenerate continuous action of $\mathbb C[[\mathfrak n^*]]$ on $\pi/ \mathbf{CJ}_s^{\infty}(\pi)$ i.e. a surjective continuous morphism from $\mathbb{C}[[\mathfrak n^*]]\widehat{\otimes}(\pi/\mathbf{CJ}_s^{\infty}(\pi))$ to $\pi/\mathbf{CJ}_s^{\infty}(\pi)$.
\end{lemma}

\begin{proof}
Let $\lambda=\pi/\mathbf{CJ}^{\infty}_s(\pi)$. It follows from Proposition \ref{prop nondeg schwartz} that the action map:
\[   \mathcal S(\mathfrak n) \widehat{\otimes} \lambda \rightarrow \lambda 
\]
is continuous and non-degenerate. We shall consider the action after applying the Fourier transform discussed in Section \ref{ss fourier transform}, and so we have a continuous map $s$ from $\mathcal S(\mathfrak n^*)\widehat{\otimes}\lambda$ to $\lambda$.

Now $\mathcal S(\mathfrak n^*\setminus \left\{ 0 \right\}).\lambda$ is zero by Lemma \ref{lem inclusion}. Then one has a well-defined map $s_0$ from $\mathbb C[[\mathfrak n^*]]\widehat{\otimes} \lambda$ to $\lambda$ given by:
\[   f \otimes x \stackrel{s_0}{\mapsto}  \pi(\widetilde{f})x  ,
\]
where $\widetilde{f}$ is a lift of $f$ in $\mathcal S(\mathfrak n^*)$. It again follows from Lemma \ref{lem inclusion} that the action is well-defined, i.e. independent of a choice of $\widetilde{f}$.

Now, by \cite[Proposition 43.9]{Tr06}, the natural map $\mathrm{pr}$ from $\mathcal S(\mathfrak n^*) \widehat{\otimes} \pi$ to $\mathbb C[[\mathfrak n^*]] \widehat{\otimes} \pi$ is surjective. Since $s_0 \circ \mathrm{pr}=s$ and $s$ is also surjective, we must have $s_0: \mathbb C[[\mathfrak n^*]]\widehat{\otimes}\lambda \rightarrow \lambda$ is again also surjective, as desired. 
\end{proof}

\begin{proposition} \label{prop convergence of nilpotent action}
Let $\pi$ be in $\mathrm{Rep}^{\infty, F}(N)$. There exists a countable family $\left\{ q_j \right\}$ of seminorms defining the topology on $\pi/\mathbf{CJ}_s^{\infty}(\pi)$ such that for each fixed $q_j$, there exists a sufficiently large integer $k_j>0$ such that the seminorm $q_j$ vanishes on $\overline{\mathfrak n^{k}. \pi}/\mathbf{CJ}_s^{\infty}(\pi)$ for all $k >k_j$. In particular, for any $r>k>k_j$, the induced seminorm $\widetilde{q}_k$ on $\overline{\mathfrak n^k.\pi}/\overline{\mathfrak n^r.\pi}$ vanishes. 
\end{proposition}

\begin{proof}
 It follows from discussions in Section \ref{ss topology on power series} that the seminorms $\left\{ p_i \right\}_{i \in \mathbb Z_{\geq 0}}$ defining the Fr\'echet topology of $\mathbb C[[\mathfrak n^*]]$ can be chosen as
\[  p_i(\sum_{i_1, \ldots, i_n \geq 0} c_{i_1,\ldots, i_r}X_1^{i_1}\ldots X_n^{i_n}) := \mathrm{max}_{i_1+\ldots+i_n\leq i}\left\{ |c_{i_1, \ldots, i_n}| \right\} ,
\]
where $X_1, \ldots, X_r$ is a fixed basis for $\mathfrak n^*$. Thus the projective seminorm $\left\{ p_i \otimes q_j \right\}_{i,j \in \mathbb Z_{\geq 1}}$ defines the topology on $\mathbb C[[\mathfrak n^*]]\widehat{\otimes} \lambda$. 

Let $\lambda=\pi/\mathbf{CJ}_s^{\infty}(\pi)$. By  Lemma \ref{lem non-dgenerate quotient}, we have a continuous surjection:
\[   \mathbb C[[\mathfrak n^*]] \widehat{\otimes} \lambda \rightarrow \lambda .
\]
Then, by the open mapping theorem, $\lambda$ can be realized as a Hausdorff quotient of $\mathbb C[[\mathfrak n^*]]\widehat{\otimes} \lambda$. Recall that the quotient seminorms $\widetilde{p_i\otimes q_j}$ define the topology of $\lambda$ in Section \ref{ss quotient topology}. For fixed $i$ (and $j$), by Lemma \ref{lem intersect jacquet}, we can choose a sufficiently large integer $k$ such that $\mathfrak n^k.\mathbb C[[\mathfrak n^*]]$ vanishes on $p_i$. Note that $\mathfrak n^k.(\mathbb C[[\mathfrak n^*]]\widehat{\otimes}\lambda)$ surjects onto $\mathfrak n^k.\lambda/\mathbf{CJ}_{s}^{\infty}(\pi)$. Then, for $x \in \mathfrak n^k.\lambda$, we can choose a representative $\widetilde{x}$ such that $(p_i\otimes q_j)(\widetilde{x})=0$ and so $\widetilde{p_i\otimes q_j}(\widetilde{x})=0$. Now, taking the continuity on the closure, we have that $\widetilde{p_i\otimes q_j}$ vanishes on $\overline{\mathfrak n^k.\lambda}=\overline{\mathfrak n^k.\pi}/\mathbf{CJ}_s^{\infty}(\pi)$, as desired.
\end{proof}

\begin{corollary} \label{cor converging seq 0}
Let $\pi$ be in $\mathrm{Rep}^{\infty,F}(N)$. Let $(x_k)_{k=1}^{\infty}$ be a sequence in $\pi/\mathbf{CJ}_{s}^{\infty}(\pi)$ such that $x_k \in \overline{\mathfrak n^k.\pi}/\mathbf{CJ}^{\infty}_s(\pi)$ for each positive integer $k$. Then the sequence $(x_k)$ converges to $0$.
\end{corollary}

\begin{proof}
Let $\left\{ q_i \right\}_{i \in \mathbb Z_{\geq 1}}$ be a family of seminorms defining the topology of $\mathbf{CJ}^{\infty}(\pi)$ in Proposition \ref{prop convergence of nilpotent action}. It follows from Section \ref{ss quotient topology} that the topology of $\mathbf{CJ}^{\infty}(\pi)$ is also defined by the metric: for $v, v' \in \pi/\mathbf{CJ}_s^{\infty}(\pi)$,
\[    d(v,v')= \sum_{j=1}^{\infty}\frac{1}{2^j} \frac{q_j(v-v')}{1+q_j(v-v')} .
\]
Now, from the previous proposition, we have that $d(x_k,0) \rightarrow 0$ and so $x_k\rightarrow 0$ as desired. 
\end{proof}

\section{Stability Theorem} \label{s n stable}
Let $N$ be the unipotent radical of a parabolic subgroup of $G$, and let $\mathfrak n=\mathrm{Lie}(N)$.

\begin{definition}
 For $\pi$ in $\mathrm{Rep}^{\infty, F}(N)$, $\pi$ is said to be {\it $\mathfrak n$-stable} if $\overline{\mathfrak n.\pi}=\pi$. 
\end{definition}

We sketch the intuitive idea on the stability result of Corollary \ref{cor maximal stable quotient} below. A main idea is to ``glue'' non-trivially $\mathbf{CJ}^{\infty}(\pi)$ with a module $\omega$ with the trivial $\mathfrak n$-action. The way to glue modules to form a module $\pi'$ is in the sense that $\mathbf{CJ}_s^{\infty}(\pi')$ is $\omega$. The module $\omega$ also has to be the submodule of $\overline{\mathfrak n^k.\pi'}$. Now one picks an element $v$ in $\omega$ and constructs a sequence converging to $v$, and finally shows from Proposition \ref{prop convergence of nilpotent action} that this sequence also has to converge to $0$.
 

\begin{corollary} \label{cor maximal stable quotient}
Let $\pi$ be in $\mathrm{Rep}^{\infty, F}(N)$. Then $\mathbf{CJ}_s^{\infty}(\pi)$ is $\mathfrak n$-stable.
\end{corollary}

\begin{proof}
Let $\lambda=  \bigcap_{k=0}^{\infty} \overline{\mathfrak n^k. \pi}=\mathbf{CJ}_s^{\infty}(\pi)$. Let
\[   \pi'= \pi/\overline{\mathfrak n.\lambda} .
\]
Let $\mathrm{pr}: \pi \rightarrow \pi'$ be the projection map. Let $v$ be an arbitrary element in $\mathrm{pr}(\lambda)$. We proceed to show that $v=0$. \\

\noindent
{\it Claim 1:} We fix a basis $\left\{ X_1,\ldots, X_n\right\}$ for $\mathfrak n$. Then, for each $k \in \mathbb Z_{\geq 1}$, there exists a sequence of sets $\left\{ x^k_1, \ldots, x^k_n\right\}$ of vectors in $\mathrm{pr}(\overline{\mathfrak n^k.\pi})$ such that 
\[   X_1x^k_1+\ldots + X_nx^k_n  \longrightarrow v .
\]

\noindent
{\it Proof of Claim 1:} 
We fix a non-negative integer $k$. By the definition and closedness of $\lambda$, we have:
\[   \overline{\mathfrak n.\lambda} \subset \lambda \subset \overline{\mathfrak n^{k+1}.\pi} .
\]
It is clear that 
\[   \mathfrak n. \mathrm{pr}(\overline{\mathfrak n^k.\pi}) \subset \overline{\mathrm{pr}(\mathfrak n^{k+1}.\pi)} \quad \Longrightarrow \quad  \overline{ \mathfrak n. \mathrm{pr}(\overline{\mathfrak n^k.\pi})} \subset \overline{\mathrm{pr}(\mathfrak n^{k+1}.\pi)}  .
\]
We also have
\[   \mathrm{pr}(\mathfrak n^{k+1}.\pi) \subset \mathfrak n.\mathrm{pr}(\overline{\mathfrak n^{k}.\pi}) \quad \Longrightarrow \quad \overline{\mathrm{pr}(\mathfrak n^{k+1}.\pi)} \subset \overline{\mathfrak n.\mathrm{pr}(\overline{\mathfrak n^k.\pi})} .
\]
Combining with $\overline{\mathrm{pr}(\overline{\mathfrak n^{k+1}.\pi})}=\overline{\mathrm{pr}(\mathfrak n^{k+1}.\pi)}$, this implies that $\overline{\mathrm{pr}(\overline{\mathfrak n^{k+1}.\pi})} =\overline{\mathfrak n.\mathrm{pr}(\overline{\mathfrak n^k.\pi})}$ and so
\[  \mathrm{pr}(\lambda) \subset \overline{\mathfrak n.\mathrm{pr}(\overline{\mathfrak n^k.\pi})} .
\]

This implies that, for each $k$, there exists a sequence of sets $\left\{ x_1^k,\ldots, x_n^k \right\}$ in $\mathrm{pr}(\overline{\mathfrak n^k.\pi})$ such that $d(X_1x^k_1+\ldots+X_nx^k_n, v)<\frac{1}{2^k}$. This provides the desired sequence.
\\

Claim 1 above provides a sequence of vectors in $\mathrm{pr}(\overline{\mathfrak n^k.\pi})$ converging to $v$. However, we need to make some good choices, for which we shall need Claim 2 below. \\

\noindent
{\it Claim 2:} Let $(x_1^k, \ldots, x_n^k)_{k \in \mathbb Z_{\geq 1}}$ be a sequence of the vectors in the previous claim. Suppose $(y^k_1, \ldots, y^k_n)_{k \in \mathbb Z_{\geq 1}}$ form another sequence of vectors such that, for each $k$ and $i$, $x_i^k$ and $y_i^k$ have the same image in the quotient $\pi/\lambda$. Then 
\[  X_1y_1^k+\ldots +X_ny_n^k \longrightarrow v  .
\]
\noindent
{\it Proof of Claim 2:} Note that 
\[ (X_1x_1^k+\ldots+X_nx_n^k) -(X_1y_1^k+\ldots+X_ny_n^k)=X_1(x_1^k-y_1^k)+\ldots +X_n(x_n^k-y_n^k) \] 
is zero because $x_1^k-y_1^k, \ldots, x_n^k-y_n^k$ are in $\lambda$ and so the above expression is in $\mathfrak n.\lambda \subset \overline{\mathfrak n.\lambda}$.  In other words, $(X_1x_1^k+\ldots +X_nx_n^k)-(X_1y_1^k+\ldots +X_ny_n^k)=0$ and so this gives the claim. \\

Now back to the proof, from Claim 1, we can find a sequence of $x_1^k, \ldots, x_n^k \in \mathrm{pr}(\overline{\mathfrak n^k.\pi})$ such that $X_1x_1^k+\ldots+X_nx_n^k$ converges to $v$. 

On the other hand, let $\mathrm{pr}':\pi' \rightarrow  \pi/\lambda$ be the projection. We now consider the projections of $x_1^k, \ldots, x_n^k$ 
\[  \mathrm{pr}'(x_1^k), \ldots, \mathrm{pr}'(x_n^k) 
\]
to $\pi/\mathbf{CJ}_s^{\infty}(\pi)$. By Corollary \ref{cor converging seq 0}, 
\[  \mathrm{pr}'(x_1^k), \ldots, \mathrm{pr}'(x_n^k) 
\]
converge to $0$. Hence, from the quotient topology, we can find representatives $y_1^k,\ldots ,y_n^k$ in $\pi'$ of respective  $\mathrm{pr}'(x_1^k),\ldots, \mathrm{pr}'(x_n^k)$ so that each sequence $(y_i^k)_k$ converges to $0$. In particular, we also have:
\[   X_1y_1^k+\ldots +X_ny_n^k \rightarrow 0
\]
Now by Claim 2, such the sequence also converges to $v$. Hence, $v=0$.
\end{proof}

As a consequence  of the $\mathfrak n$-stability, we also have the stability of the Casselman-Jacquet submodule functor.

\begin{corollary} \label{cor cj functor stablize}
Let $\pi$ be in $\mathrm{Rep}^{\infty, F}(N)$. Then $\mathbf{CJ}^{\infty}_{s,P}\circ \mathbf{CJ}^{\infty}_{s,P}(\pi)=\mathbf{CJ}^{\infty}_{s,P}(\pi)$.
\end{corollary}

\begin{proof}
Corollary \ref{cor maximal stable quotient} implies that $\overline{\mathfrak n^k.\mathbf{CJ}_s^{\infty}(\pi)}=\mathbf{CJ}^{\infty}_s(\pi)$ for any $k \in \mathbb Z_{\geq 0}$, and this implies the corollary.
\end{proof}

\section{ Surjection onto the Casselman-Jacquet quotient} \label{s exact of cj}

We now consider the natural map from $\pi$ to $\mathbf{CJ}^{\infty}(\pi)$. We remark that such map is continuous by the universal property of the projective tensor product and the continuity of projections $\pi\rightarrow \pi/\overline{\mathfrak n^k.\pi}$ ($k \in \mathbb Z_{\geq 1}$). The key idea of proving such the natural map is surjective is that for the given element $(x_1,x_2, \ldots)$ in $\mathbf{CJ}^{\infty}_s(\pi)$, we choose suitable representatives $\widetilde{x}_i$ in $\pi$ with the property that the projection of $\widetilde{x}_i$ onto $\pi/\overline{\mathfrak n^i.\pi}$ is $x_i$. The good choices are determined by the seminorms in Proposition \ref{prop convergence of nilpotent action} and allow one to show that the sequence converges.

\begin{proposition} \label{prop surj cj quotient}
Let $P$ be a parabolic subgroup of $G$ and let $\pi$ be in $\mathrm{Rep}^{\infty, F}(P)$. 
\begin{enumerate}
\item The functor $\mathbf{CJ}^{\infty}$ is well-defined i.e. $\mathbf{CJ}^{\infty}(\pi)$ is of moderate growth. 
\item The natural map from $\pi$ to $\mathbf{CJ}^{\infty}(\pi)$ is surjective.
\item As representations in $\mathrm{Rep}^{\infty,F}(P)$, $ \pi/\mathbf{CJ}_s^{\infty}(\pi) \cong \mathbf{CJ}^{\infty}(\pi)$.
\end{enumerate}
\end{proposition}

\begin{proof}
Note that (1) and (3) follow from (2) and the open mapping theorem, and the fact that $\mathbf{CJ}^{\infty}_s(\pi)$ is the kernel of the natural map from $\pi$ to $\mathbf{CJ}^{\infty}(\pi)$. 

We now prove (2). Let $(x_1, x_2, \ldots)$ be in $\varprojlim_k \pi/\overline{\mathfrak n^k.\pi}$. We shall construct a convergent sequence $(y_n)$ in $\pi$ such that the limit of $y_n$ is projected to $(x_1, x_2, \ldots)$. 

Let $\left\{ p_i \right\}_{i \in \mathbb Z_{\geq 0}}$ be a countable family of seminorms defining the topology of $\pi$ satisfying:
\[ p_1 \leq p_2 \leq \ldots \]
By Proposition \ref{prop convergence of nilpotent action}, for each $i \in \mathbb Z_{\geq 0}$, there exists an integer $k_i$ such that the seminorm induced from $p_i$ vanishes on $\overline{\mathfrak n^{k_i}.\pi}/\overline{\mathfrak n^r.\pi}$ for any $r >k_i$. We may and shall further assume that $k_1<k_2<k_3<\ldots$. 

Let $\iota_{k_2}:\pi\rightarrow \pi/\overline{\mathfrak n^{k_2}.\pi}$. Let $y_1$ be an element in $\pi$ such that the projection of $y_1$ from $\pi$ to $\pi/\overline{\mathfrak n^{k_1}.\pi}$ is $x_{k_1}$. Then the element $x_{k_2}-\iota_{k_2}(y_1)$ is in $\overline{\mathfrak n^{k_1}.\pi}/\overline{\mathfrak n^{k_2}.\pi}$. Since the induced seminorm of $p_1$ vanishes on $\overline{\mathfrak n^{k_1}.\pi}/\overline{\mathfrak n^{k_2}.\pi}$, we can find an element $y_2$ in $\overline{\mathfrak n^{k_2}.\pi}$ such that the following hold: $p_1(y_2) \leq \frac{1}{2}$ and; the projection of $y_2$ from $\pi$ to $\pi/\overline{\mathfrak n^{k_2}.\pi}$ is $x_{k_2}-\iota_{k_2}(y_1)$.

We similarly and inductively construct elements $y_s$ such that the following conditions hold:
\begin{enumerate}
\item $p_s(y_{s+1}) \leq \frac{1}{2^s}$
\item the projection of $y_{s+1}$ from $\pi$ to $\pi/\overline{\mathfrak n^{k_s+1}.\pi}$ is $x_{k_{s+1}}-\iota_{k_{s+1}}(y_1-\ldots-y_s)$. 
\end{enumerate}

In
\[  d(y_{s+1}) =\sum_{i=1}^s\frac{1}{2^i}\cdot \frac{p_i(y_{s+1})}{1+p_i(y_{s+1})} + \sum_{i=s+1}^{\infty}\frac{1}{2^i}\cdot \frac{p_i(y_{s+1})}{1+p_i(y_{s+1})} , \]
the first summation is bounded by $\frac{1}{2^{s-1}}$ and the second summation is also bounded by $\frac{1}{2^{s}}$. Thus, $\{\sum_{i=1}^sy_i\}$ forms a Cauchy sequence and so converges in $\pi$. It remains to observe that the element $\sum_{i=1}^{\infty}y_i$ under the map from $\pi$ to $\mathbf{CJ}_P^{\infty}(\pi)$ is $(x_1,x_2, \ldots)$. 
\end{proof}

\section{Epiness or surjectivity under Casselman-Jacquet functors} \label{s epi cj sub functor}

\subsection{Epiness for Casselman-Jacquet submodules} \label{ss epiness of cj}

     Let $N$ be a unipotent radical of a parabolic subgroup of $G$. Note that, in $\mathrm{Rep}^{\infty, F}(N)$, a morphism $f: \pi_1 \rightarrow \pi_2$ is epi if and only if the image of $f$ is dense in $\pi_2$; and $f$ is monic if and only if $f$ is injective.

\begin{proposition} \label{prop monic epi property}
Let $\pi_1, \pi_2$ be in $\mathrm{Rep}^{\infty, F}(N)$.
\begin{enumerate}
\item If $f: \pi_1 \rightarrow \pi_2$ is monic (and so injective), then $\mathbf{CJ}^{\infty}_s(f)$ is also monic (and so injective).
\item If $f: \pi_1 \rightarrow \pi_2$ is epi, then  $\mathbf{CJ}_s ^{\infty}(f)$ is also epi.
\end{enumerate}
\end{proposition}  

\begin{proof}
(1) is clear as the induced map is simply restricted to $\mathbf{CJ}_s^{\infty}(\pi_1)$ and so is still injective.

We now consider (2). Let $f: \pi_1\rightarrow \pi_2$ be an epimorphism in $\mathrm{Rep}^{\infty, F}(N)$. Let 
\[   \lambda_1=\bigcap_{k=0}^{\infty} \overline{\mathfrak n^k.\pi_1}, \quad \lambda_2=\bigcap_{k=0}^{\infty} \overline{\mathfrak n^k.\pi_2} .
\]
Since $f(\overline{\mathfrak n^k.\pi_1}) \subset \overline{\mathfrak n^k.\pi_2}$ for all $k$, we have:
\[ f(\lambda_1)\subset \lambda_2 .
\]
In order to show that $\mathbf{CJ}_s ^{\infty}(f)$ is epi, it suffices to show that $\overline{f(\lambda_1)}=\lambda_2$ or equivalently $\lambda_2/\overline{f(\lambda_1)}=0$. 

We consider an induced map $\widetilde{f}$ from $\pi_1/\lambda_1$ to $\pi_2/\overline{f(\lambda_1)}$. Let $x \in \lambda_2/\overline{f(\lambda_1)} \subset \pi_2/\overline{f(\lambda_1)}$. We want to show that $x=0$. \\

\ \\

\noindent
{\it Claim:} For any $k \in \mathbb Z_{\geq 0}$, there exists $\widetilde{x}_k \in \overline{\mathfrak n^k.\pi_1}/\lambda_1$ such that $d(f(\widetilde{x}_k),x)<\frac{1}{2^k}$, where $d$ is the function defining the metric on $\pi_2/\overline{f(\lambda_1)}$. 

\noindent
{\it Proof of Claim:} By the $\mathfrak n$-stable property on $\lambda_2$ (Corollary \ref{cor maximal stable quotient}), for any $k \geq 0$,
\[   x \in \overline{\mathfrak n^k.\lambda_2}/\overline{f(\lambda_1)} .
\]

Thus, there exists $x_k$ in $(\mathfrak n^k.\lambda_2)/\overline{f(\lambda_1)}$ such that 
\[    d(x_k,x)<\frac{1}{2^k} .
\]
Then, one takes $\widetilde{x}_k$ to be any lift of $x_k$ in $\overline{\mathfrak n^k.\pi_1}/\lambda_1$ and so satisfies the required property of the claim. \\

We go back to the proof of the proposition. It follows from Proposition \ref{prop convergence of nilpotent action} that the sequence $(\widetilde{x}_k)$ in the Claim converges to $0$. Hence, $x=\lim_k f(\widetilde{x}_k)=0$, as desired.
\end{proof}

\subsection{Surjectivity and epiness under Casselman-Jacquet quotient functor}

We now show a preservation of surjectivity under the Casselman-Jacquet quotient functor. The key is to deduce from the surjectivity result in Proposition \ref{prop surj cj quotient}.

\begin{proposition} \label{prop surj cj quotient functor}
Let $f: \pi_1 \rightarrow \pi_2$ be a surjective (resp. epi) morphism in $\mathrm{Rep}^{\infty,F}(G)$. Let $P$ be a parabolic subgroup of $G$. Then $\mathbf{CJ}_P^{\infty}(f)$ is also surjective (resp. epi).
\end{proposition}

\begin{proof}
We only prove the statement for the surjectivity part, and the epiness part is quite similar by taking the closure. Let $f: \pi_1\rightarrow \pi_2$ be surjective. Let $y \in \mathbf{CJ}^{\infty}(\pi_2)$. Proposition \ref{prop surj cj quotient} implies that there exists $x \in \pi_2$ such that 
\[ y=(x, x, \ldots ) \in \varprojlim_k \pi_2/\overline{\mathfrak n^k.\pi_2} .
\]
Here each $x$ is considered as the projection of $x$ from $\pi_2$ to the corresponding $\pi_2/\overline{\mathfrak n^k.\pi_2}$. 

The surjectivity of $f$ implies that there exists $\widetilde{x} \in \pi_1$ such that $f(\widetilde{x})=x$. The element $(\widetilde{x}, \widetilde{x}, \ldots)$ in $\mathbf{CJ}_P^{\infty}(\pi_1)$ is then mapped to $y$ under the map $\mathbf{CJ}_P^{\infty}(f)$, as desired.
 \end{proof}

\subsection{Non-injectivity of Casselman-Jacquet quotient functor}

\begin{example} \label{ex small seminorm family} (Failure of preserving injectivity of $\mathbf{CJ}^{\infty}$)
Let $C^{\infty}(\mathbb R^{\times})$ be the space of smooth functions from $\mathbb R^{\times}$ to $\mathbb C$. We equip the space with the topology induced from the family of seminorms given by: for $r, s \in \mathbb Z_{\geq 0}$, and for $f \in C^{\infty}(\mathbb R^{\times})$,
\[  p_{r,s}(f) := \mathrm{sup}_{x \in \mathbb R^{\times}}\left\{ x^{r+s}\frac{d^sf}{dx^s}(x) \right\} .
\]
Let $\pi_1$ be the subspace of $C^{\infty}(\mathbb R^{\times})$ such that for any $r,s \in \mathbb Z_{\geq 0}$, $p_{r,s}(f)<\infty$ and 
\[  \lim_{x\rightarrow 0^+} x^{r+s}\frac{d^sf}{dx^s} =\lim_{x\rightarrow 0^-}x^{r+s}\frac{d^sf}{dx^s} \quad \mbox{exists} . \]
The conditions on the limit at $0$ imply that for $f \in \pi_1$, $x^{r+s}\frac{d^sf}{dx^s}$ has to be uniformly continuous in a neighborhood of $0$. Using the standard fact that the limit of any uniform convergent sequence of uniformly continuous functions on an open interval in $\mathbb R^{\times}$ is still uniformly continuous, we have that the space $\pi_1$ is complete under the seminorms $p_{r,s}$. 

Let $M_2$ be the mirabolic subgroup of $\mathrm{GL}_2(\mathbb R)$. We equip $\pi_1$ with the $M_2$-action given by:
\[  \left( \begin{pmatrix} 1 & a \\ 0 & 1 \end{pmatrix}.f \right)(x)=e^{-2\pi\sqrt{-1}xa} f(x), \quad  \left( \begin{pmatrix} t & \\ & 1     \end{pmatrix}.f\right)(x)=f(tx) .
\]
It follows from definitions that the $M_2$-representation $\pi_1$ is smooth and of moderate growth.

Let $\pi_2=\mathcal S(\mathbb R)$ be as in Example \ref{ex schwartz function M2}. Note that $\pi_2$ naturally sits inside $\pi_1$, and the inclusion is continuous.

Let $N=\left\{ \begin{pmatrix} 1 & x \\ 0 & 1 \end{pmatrix} : x \in \mathbb R\right\}$ and let $\mathfrak n=\mathrm{Lie}(N)$. Then $\bigcap_{k=1}^{\infty}\mathfrak n^k.\pi_1$ is the space $\mathcal S(\mathbb R^{\times})$, as vector spaces. However, $\mathcal S(\mathbb R^{\times})$ does not form a complete subspace in $\pi_1$ under the topology defined by seminorms $p_{r,s}$. Indeed, we can consider the following example: 
\[ f_n(x)=\left\{ \begin{array}{cc} xe^{-(x^{\frac{-1}{n}}+(1-x)^{-1})} & \mbox{ if $0<x<1$} \\
                                      0   &  \mbox{if $x \geq 1$ or $x < 0$ }
                                      \end{array} \right.
\]
Note that the sequence $(f_n)$ converges to 
\[  f(x)= \left\{ \begin{array}{cc} xe^{-(1+(1-x)^{-1})} & \mbox{ if $0<x<1$} \\
                                         0 & \mbox{ if $x\geq 1$ or $x< 0$ } \end{array} \right.
\]
under the seminorm $\left\{ p_{r,s}\right\}$. Each $f_n$ is in $\mathcal S(\mathbb R^{\times})$, but their limit $f$ is not in $\mathcal S(\mathbb R^{\times})$. This explains why the space $\mathcal S(\mathbb R^{\times})$ is not complete in $\pi_1$ under the topology from $\left\{ p_{r,s}\right\}$. 

Now the restriction map to $\mathbb R^{\times}$, $\iota:\pi_2 \rightarrow \pi_1$, is well-defined and continuous. The image of $f\in \mathcal{S}(\mathbb{R^{\times}})$ in $\mathbf{CJ}_{N}^{\infty}(\pi_1)$ is zero, but the image of $f$ (extended to $0$) in $\mathbf{CJ}_{N}^{\infty}(\pi_2)$ is non-zero. Hence, $\mathbf{CJ}_{N}^{\infty}(\iota)$ is not injective.
\end{example}

\section{Transitivity of Casselman-Jacquet functors} \label{s transitive cj}

\begin{lemma} \label{lem two isomorphisms}
Let $\pi$ be in $\mathrm{Rep}^{\infty, F}(G)$. Let $Q\subset P\subset G$ be parabolic subgroups of $G$. Let $\mathfrak n_P=\mathrm{Lie}(N_P)$ and let $\mathfrak n_Q=\mathrm{Lie}(N_Q)$. Let $\sigma=\mathbf{CJ}^{\infty}_P(\pi)$. Let $\mathrm{pr}: \pi \rightarrow \sigma$ be the projection map. Then, for any $k \in \mathbb Z_{\geq 0}$,  
\[ \mathrm{pr}(\overline{\mathfrak n_Q^k.\pi}) =\overline{\mathfrak n_Q^k.\sigma}, \quad  \mbox{ and } \quad  \pi/(\overline{\mathfrak n_Q^k.\pi}) \cong \sigma/(\overline{\mathfrak n_Q^k.\sigma}) .
\]

\end{lemma}

\begin{proof}
We first show the former equality. Since $\mathrm{pr}$ is surjective (Proposition \ref{prop surj cj quotient}), we also have $ \mathrm{pr}(\mathfrak n_Q^k.\pi) =\mathfrak n_Q^k.\sigma$ for any $k \in \mathbb Z_{\geq 0}$. Since $\mathbf{CJ}_{s,P}^{\infty}(\pi) \subset \overline{\mathfrak n_Q^k.\pi}$, we have that $\mathrm{pr}(\overline{\mathfrak n_Q^k.\pi})$ is also closed. Thus, we have the former equality.


We now show the latter isomorphism. Since $\mathbf{CJ}^{\infty}_{s,P}(\pi)\subset \overline{\mathfrak n_P^k.\pi}\subset \overline{\mathfrak n_Q^k.\pi}$, 
\[  \frac{\pi}{\overline{\mathfrak n_Q^k.\pi}} \cong \frac{\pi/\mathbf{CJ}^{\infty}_{s,P}(\pi)}{\overline{\mathfrak n_Q^k.\pi}/\mathbf{CJ}^{\infty}_{s,P}(\pi)} .
\]
The latter isomorphism follows from $\sigma =\pi/\mathbf{CJ}^{\infty}_{s,P}(\pi)$ and the first equality.
\end{proof}

\begin{proposition} (Transitivity of CJ functor) \label{prop trans CJ}
Let $\pi$ be in $\mathrm{Rep}^{\infty, F}(G)$. Let $Q\subset P\subset G$ be parabolic subgroups of $G$. Then
\[   \mathbf{CJ}_Q^{\infty}\circ \mathbf{CJ}_P^{\infty}(\pi) \cong \mathbf{CJ}_Q^{\infty}(\pi) .
\]
\end{proposition}

\begin{proof}
This follows from the second assertion in Lemma \ref{lem two isomorphisms}.
\end{proof}

\section{Alternate descriptions of Casselman-Jacquet submodule functor} \label{s description intersect cj}

\subsection{Nakayama lemma for $\mathrm{Rep}^{\infty, F}(N)$}

One may regard the following proposition as an analogue of the Nakayama lemma for smooth representations over the local ring $\mathbb C[[\mathfrak n^*]]$.

\begin{proposition} \label{prop zero from stable}
Let $N$ be the unipotent radical of a parabolic subgroup of $G$ and let $\mathfrak n=\mathrm{Lie}(N)$. Let $\pi$ be in $\mathrm{Rep}^{\infty,F}(N)$. Suppose the following two conditions hold:
\begin{enumerate}
\item $\pi$ is $\mathfrak n$-stable; and
\item $\mathcal S(\mathfrak n^* \setminus \left\{ 0 \right\}).\pi=0$.
\end{enumerate}
Then $\pi=0$.
\end{proposition}

\begin{proof}


By the same argument in the proof of Lemma \ref{lem non-dgenerate quotient}, the condition (2) implies that there exists a continuous surjective map:
\[  \mathbb C[[\mathfrak n^*]] \widehat{\otimes} \pi \rightarrow \pi . \]
Let $\left\{ q_j\right\}$ be a countable family of seminorms defining the topology of $\pi$. Also, the topology on $\mathbb C[[\mathfrak n^*]]$ is determined by a family of the seminorms $\left\{ p_i \right\}$ as discussed in the proof of Proposition \ref{prop convergence of nilpotent action}.

Recall that the topology on $\mathbb C[[\mathfrak n^*]] \widehat{\otimes} \pi$ is defined in Section \ref{ss projective norm}. Let $x \in \pi$. It suffices to show that $p_i\otimes q_j(x)=0$ for all $i$ and $j$. By Lemma \ref{lem intersect jacquet}, we can now choose sufficiently large $k$ such that $p_i$ vanishes on $\mathfrak n^k.\mathbb C[[\mathfrak n^*]]$. By the given assumption, there exists a sequence $( x_n )$ in $\mathfrak n^k.\pi$ converging to $x$. Now, it suffices to show that $p_i\otimes q_j(x_n)=0$ for each $n$. By the $\mathfrak n$-stable condition that $\mathbf{CJ}^{\infty}_s(\pi)=\pi$, there exist monomials $P_1,\ldots, P_r$ in $\mathbb C[\mathfrak n^*]$ with degree $k$ and elements $y_1,\ldots, y_r$ in $\pi$ such that
\[  x_n=P_1.y_1+\ldots +P_r.y_r  .\]
 
 Note that 
\[   (p_i\otimes q_j)(\sum_{t=1}^r P_t\otimes y_t) =0 . \]
Recall that, by the open mapping theorem, we can identify $\pi$ with the quotient 
\[(\mathbb C[[\mathfrak n^*]] \widehat{\otimes}\pi)/\kappa , \] 
where $\kappa$ is the kernel of the quotient map from $\mathbb C[[\mathfrak n^*]]\widehat{\otimes}\pi$ to $\pi$. Since $P_1\otimes y_1+\ldots +P_r\otimes y_r$ projects to $x_n$ by the above action map, the quotient seminorm $\widetilde{p_i \otimes q_j}$ on $x_n$ is zero. Since our choices of $i$ and $j$ are arbitrary, this implies that $x_n$ vanishes for all the seminorms $\widetilde{p_i\otimes q_j}$ defining the topology on $\pi$. Thus, $x_n=0$ and so $x=0$.
\end{proof}

\subsection{Casselman-Jacquet submodules and Schwartz algebra actions}

\begin{corollary} \label{cor closedness and stable prop}
Let $\pi$ be in $\mathrm{Rep}^{\infty, F}(N)$.  Then  
\[ \mathbf{CJ}_{s, N}^{\infty}(\pi) =  \overline{\mathcal S(\mathfrak n^* \setminus \left\{ 0 \right\}).\pi} = \overline{\bigcap_{k=0}^{\infty} \mathfrak n^k.\pi} .  \]
\end{corollary}

\begin{proof}
The inclusion $\overline{\mathcal S(\mathfrak n^*\setminus \left\{ 0 \right\}). \pi} \subset \mathbf{CJ}^{\infty}_s(\pi)$ follows from Lemma \ref{lem inclusion}. For another inclusion, one first considers the quotient $\lambda:=\mathbf{CJ}^{\infty}_s(\pi)/\overline{\mathcal S(\mathfrak n^*\setminus \left\{ 0 \right\}). \pi}$. Note that $\lambda$ satisfies the assumption (1) in Proposition \ref{prop zero from stable} by Corollary \ref{cor maximal stable quotient}, and $\lambda$ satisfies the assumption (2)  in Proposition \ref{prop zero from stable} from the definition of $\lambda$. Thus Proposition \ref{prop zero from stable} now implies the first equality in the corollary.

We now prove the second equality. Let $\omega=\mathcal S(\mathfrak n^*\setminus \left\{ 0 \right\}).\pi$. Note that, by Lemma \ref{lem n-invariant under schwartz}, $\mathfrak n^k.\mathcal S(\mathfrak n^* \setminus \left\{ 0\right\})=\mathcal S(\mathfrak n^*\setminus \left\{ 0 \right\})$ and so $\mathfrak n^k.\omega=\omega$. Thus 
\[  \omega =\mathfrak n^k.\omega \subset \bigcap_{k=1}^{\infty} \mathfrak n^k. \pi \subset \bigcap_{k=1}^{\infty} \overline{\mathfrak n^k.\pi} .
\]
Now one takes the closure, and the first equality forces all inclusions after taking closures to be equal. This shows the second equality of the corollary.
\end{proof}

\begin{corollary} \label{cor cj sub max stable}
Let $\pi$ be in $\mathrm{Rep}^{\infty,F}(N)$.   Then $\mathbf{CJ}^{\infty}_{s,P}(\pi)$ is the maximal $\mathfrak n$-stable closed subpsace of $\pi$.
\end{corollary}

\begin{proof}
Let $\pi'$ be the maximal $\mathfrak n$-stable closed subspace in $\pi$. Now one proceeds with the same proof as in Corollary \ref{cor closedness and stable prop} to obtain that $\pi'=\overline{\mathcal S(\mathfrak n^*\setminus \left\{0 \right\}).\pi}$. Now the corollary follows from the first equation in Corollary \ref{cor closedness and stable prop}.
\end{proof}

\begin{corollary} \label{cor jacquet intersect for schwartz}
We have
\[  \bigcap_{k =1}^{\infty} \mathfrak n^k.\mathcal S(\mathfrak n) =\bigcap_{k=1}^{\infty} \mathcal C_k(\mathfrak n) .
\]
\end{corollary}

\begin{proof}
The inclusion $\subset$ follows from Lemma \ref{lem intersect jacquet}. It follows from Lemma \ref{lem n-invariant under schwartz}  that $\bigcap_{k=1}^{\infty}\mathcal C_k(\mathfrak n)$ is $\mathfrak n$-stable. Thus, by Corollaries \ref{cor closedness and stable prop} and \ref{cor cj sub max stable}, $\bigcap_{k=1}^{\infty}\mathcal C_k(\mathfrak n) \subset \overline{\bigcap_{k=1}^{\infty}\mathfrak n^k.\mathcal S(\mathfrak n) }$. The inclusion $\overline{\bigcap_{k=1}^{\infty}\mathfrak n^k.\mathcal S(\mathfrak n)}\subset\bigcap_{k=1}^{\infty}\mathcal C_k(\mathfrak n)$ follows from Lemma \ref{lem intersect jacquet} and the space $\bigcap_{k=1}^{\infty}\mathcal C_k(\mathfrak n)$ is closed. It remains to show that $\bigcap_{k=1}^{\infty} \mathfrak n^k.\mathcal S(\mathfrak n)$ is closed. Suppose not. Let $\lambda:= \overline{\bigcap_{k=1}^{\infty} \mathfrak n^k.\mathcal S(\mathfrak n)}$. Then we must have $\bigcap_{k=1}^{\infty}\mathfrak n^k.\lambda \subsetneq \lambda$, and so $\mathfrak n^k.\lambda \subsetneq \lambda$ for some $k$. However, this contradicts  Lemma \ref{lem n-invariant under schwartz}. Thus $\lambda=\bigcap_{k=1}^{\infty} \mathfrak n^k.\mathcal S(\mathfrak n)$.
\end{proof}




\part{Globalization and comparison with Harish-Chandra modules} \label{part cj casselman wallach rep}

Let $G$ be a reductive Lie group and let $P$ be a parabolic subgroup of $G$. Let $K$ be a maximal compact subgroup of $G$. By conjugating $K$ with an element in $G$ if necessary, we assume that all parabolic subgroups we consider below are chosen such that $K \cap P$ is still a maximal compact subgroup of $P$. Let $\mathfrak g=\mathrm{Lie}(G)$, $\mathfrak n=\mathrm{Lie}(N_P)$, $\mathfrak m=\mathrm{Lie}(M_P)$, $\mathfrak p=\mathrm{Lie}(P)$, and let $K_P=K\cap P$. We shall keep using all these notation for the entire Part \ref{part cj casselman wallach rep}.

\section{Representations of parabolic subgroups} \label{s prelim gK mod}



\subsection{Admissibility}

 A $(\mathfrak m, K_P)$-module $\pi$ is said to be {\it admissible} if for any irreducible finite-dimensional representation $\tau$ of $K_P$, $\mathrm{dim}~\mathrm{Hom}_{K_P}(\tau, \pi)<\infty$. A representation $\pi$ in $\mathrm{Rep}^{\infty,F}(M_P)$ is said to be {\it admissible} if the space of $K_P$-finite vectors of $\pi$ forms an admissible $(\mathfrak m, K_P)$-module.

\subsection{Categories for parabolic subgroups}

\begin{definition} \label{def P category}
For $k \in \mathbb Z_{\geq 1}$, let $\mathcal{HC}_{k}(P)$ be the category of $\mathfrak p$-modules $\pi$ satisfying the following properties:
\begin{enumerate}
\item $\pi$ is an admissible $(\mathfrak m, K_P)$-module of finite length; and
\item $\mathfrak n^k$ acts trivially on $\pi$.
\end{enumerate}
The morphisms are $\mathfrak p$-equivariant and $K_P$-equivariant maps.
\end{definition}

We remark that some classical contents use the condition of "finitely-generated" over the universal enveloping algebra $U(\mathfrak m)$ of $\mathfrak m$, and it is indeed equivalent to the condition of "finite-length" (see e.g. \cite[Page 394]{Ca89}, \cite[Chapter 4]{Wa88}, \cite[Theorem 2.2.2]{AGS15a}). We shall write $\mathcal{HC}(G)$ for $\mathcal{HC}_1(G)$, and the modules in $\mathcal{HC}(G)$ are sometimes called {\it Harish-Chandra modules}.

\begin{definition} \label{def P CW category}
For $k \in \mathbb Z_{\geq 1}$, let $\mathcal{CW}_{ k}(P)$ be the category of smooth Fr\'echet representations $\pi$ of $P$ such that the space of $K_P$-finite vectors of $\pi$ forms a $\mathfrak p$-module in $\mathcal{HC}_k(P)$.  
The morphisms are $P$-equivariant continuous maps.
\end{definition}

We shall also write $\mathcal{CW}(G)$ for $\mathcal{CW}_{1}(G)$. This category $\mathcal{CW}(G)$ is sometimes referred to as the Casselman-Wallach category, and representations in $\mathcal{CW}(G)$ are sometimes called {\it Casselman-Wallach representations}.

For a Casselman-Wallach representation $\pi$ of $G$, denote by $\pi_K$ the space of $K$-finite vectors in $\pi$, regarded as a $(\mathfrak g, K)$-module. For a Harish-Chandra $(\mathfrak g, K)$ module, denote by $\pi^{\infty}$ to be the Casselman-Wallach globalization of $\pi$. The functors determine the Casselman-Wallach equivalence of the two categories:

\begin{theorem} \cite{Ca89, Wa92, BK14} \label{thm CW globalization}
There is a categorical equivalence between $\mathcal{CW}(G)$ and $\mathcal{HC}(G)$.
\end{theorem}

\section{Extended Casselman-Wallach globalization} \label{s rep para}

In this section, we generalize Theorem \ref{thm CW globalization} to the categories of parabolic subgroups in Section \ref{s prelim gK mod}. A somewhat subtle point in the generalization is the extension of the continuous action of the nilpotent algebras, which also relies on the functoriality of the Casselman-Wallach globalization. A key result in this section is the abelian structure in Corollary \ref{cor cw abelian category}.

\subsection{Continuity under nilpotent actions}

\begin{lemma} \label{lem converge nilpotent action}
 Let $\pi$ be in $\mathcal{HC}_k(P)$.  Regarding $\pi$ as a $(\mathfrak m, K_P)$-representation, let $\pi^{\infty}$ be its Casselman-Wallach globalization of $\pi$ as a $M_P$-representation. 
\begin{enumerate}
\item[(1)] Let $X \in \mathfrak n$. Then the map from $\pi^{\infty}$ to $\pi^{\infty}$ given by $v \mapsto X.v$ is continuous.
\item[(2)] For any $m \in M_P$ and for any $X \in \mathfrak n$, and for all $v \in \pi^{\infty}$, 
\[   m.(X.v) =\mathrm{Ad}(m)(X).(m.v) .
\]

\end{enumerate}
\end{lemma}

\begin{proof}
We shall also regard $\mathfrak n$ as a $M_P$-representation via the adjoint action. Then $\pi \otimes \mathfrak n$ is also a $(\mathfrak m, K_P)$-module as a tensor product of two $(\mathfrak m, K_P)$-modules. Let $T: \pi \otimes \mathfrak n\to \pi$ be the action $v\otimes X\mapsto Xv$. From the Casselman-Wallach globalization, we have the following commutative diagram:
\[  \xymatrix{   \pi \otimes \mathfrak n \ar[r]^T \ar[d] & \pi \ar[d] \\
                 (\pi \otimes \mathfrak n)^{\infty} \ar[r]^{\widetilde{T}} & \pi^{\infty}
} ,
\]
with the vertical maps to be the map from the Casselman-Wallach globalization, and $\widetilde{T}$ is the morphism from the Casselman-Wallach globalization of $T$ (see \cite[Theorem 11.6.7]{Wa92} and \cite[Proposition 11.2]{BK14}). 

We also naturally have $(\pi \otimes \mathfrak n)^{\infty} \cong \pi^{\infty}\otimes \mathfrak n$ by the uniqueness of the Casselman-Wallach globalization again, and the morphism agrees on $K$-finite vectors. Note that, for a convergent sequence $v_n$ on $\pi^{\infty}$, $ v_n \otimes X$ also forms a convergent sequence in $\pi^{\infty}\otimes \mathfrak n$.   Under the continuous morphism $\widetilde{T}$, the sequence $\widetilde{T}(X\otimes v_i)=X.v_i$ is also convergent. This proves (1).

We now prove (2). We again identify $\pi^{\infty}\otimes \mathfrak n$ with $(\pi\otimes \mathfrak n)^{\infty}$. Thus, $\widetilde{T}$ commutes with the $M_P$-action and so this gives (2).
\end{proof}

\subsection{Exponentiating the Lie algebra action of $\mathfrak n$}

Let $\pi$ be in $\mathcal{CW}_{k}(P)$. We now exponentiate the action of $\mathfrak n$ for $\pi$. For $v \in \pi$, define $\pi_v$ to be the space spanned by $\mathfrak n^r.v$ for $r \in \mathbb Z_{\geq 0}$. It follows from definitions that $\pi_v$ is finite-dimensional. Let $n_v=\mathrm{dim}_{\mathbb C}(\pi_v)$. Then we have a map:
\[   \psi_v: \mathfrak n \rightarrow \mathfrak{gl}(\pi_v) \cong \mathfrak{gl}(n_v, \mathbb C)  .
\]
Since $N_P$ is simply-connected, the classical Lie theorem provides a map:
\[   \Psi_v: N_P \rightarrow \mathrm{GL}(\pi_v) .
\]
The map is well-defined, i.e. for all $n \in N_P$, $\Psi_{v_1}(n)|_{\pi_{v_1}\cap \pi_{v_2}}=\Psi_{v_2}(n)|_{\pi_{v_1}\cap \pi_{v_2}}$ for any $v_1, v_2 \in \pi$, by the uniqueness of lifting.

\begin{lemma} \label{lem phi action}
Let $\pi$ be as above. The map $\Psi: N_P \times \pi \rightarrow \pi$ given by $(n, v) \mapsto n.v$ defines a continuous map.
\end{lemma}

\begin{proof}
Fix $v_0 \in \pi$. Then
\begin{align} \label{eqn define group action}
 X\stackrel{\Psi_{v_0}}{\mapsto} \mathrm{exp}(X).v_0=(\mathrm{Id}+\psi_{v_0}(X)+\frac{\psi_{v_0}(X)^2}{2}+\ldots +\frac{\psi_{v_0}(X)^{k-1}}{(k-1)!}).v_0 
\end{align}
is continuous since each $X^i$ is continuous by Lemma \ref{lem converge nilpotent action}. Since $N$ and $\mathfrak n$ are diffeomorphic (in particular homeomorphic), we then have the map $n \in N \mapsto n.v_0$ is continuous. Thus, by the Banach-Steinhaus Theorem, the set of operators 
\[  \left\{ \Psi_{v_0}(n) : n \in \Lambda \right\},
\]
for any compact subset $\Lambda$ in $N$, is equicontinuous, and so $\Psi$ is continuous by \cite[Page 219]{Wa72}.
\end{proof}

\begin{corollary} \label{cor lifting action to P}
We use the notation in Lemma \ref{lem phi action}. Define $\psi: P \times \pi^{\infty} \rightarrow \pi^{\infty}$ given by: for $p=mn$ with $m \in M_P$ and $n \in N_P$, $(p, v) \mapsto m.(\psi(n).v)$, where the action of $m$ is the action from $\pi^{\infty}$ and the action of $\psi(n)$ is defined in Lemma \ref{lem phi action}. Then $\psi$ is continuous. 
\end{corollary}

\begin{proof}
The map $\psi$ coincides with the following natural composition:
\[   P \times \pi^{\infty} \rightarrow M_P \times N_P \times \pi^{\infty} \rightarrow \pi^{\infty} ,
\]
where the first map comes from the diffeomorphism of $P$ and $M_P \times N_P$, and the second map comes from the action maps for $M_P$ and $N_P$. Since both maps in the composition are continuous, $\psi$ is continuous. The well-definedness for the $P$-action follows from (\ref{eqn define group action}) and Lemma \ref{lem converge nilpotent action}(2).
\end{proof}

\subsection{Categorical equivalences}

 We now have the following enhancement of the Casselman-Wallach equivalence:

\begin{theorem} \label{thm equivalence of categories}  
 There is a categorical equivalence between $\mathcal{CW}_{k}(P)$ and $\mathcal{HC}_{k}(P)$.
\end{theorem}

\begin{proof}
Define $\Lambda: \mathcal{CW}_{k}(P) \rightarrow \mathcal{HC}_{k}(P)$ given by $\pi \mapsto \pi_{K_P}$. Define  $\Omega: \mathcal{HC}_{k}(P) \rightarrow \mathcal{CW}_{k}(P)$ given by $\pi \mapsto \pi^{\infty}$, where $\pi^{\infty}$ is the Casselman-Wallach globalization for $\pi$ as a $M_P$-representation. One has to check the functors are well-defined. The less trivial point is the continuity action of $N_P$ after globalization, which follows from Corollary \ref{cor lifting action to P}. Checking the two functors are inverse to each other is straightforward.
\end{proof}


A morphism $f: \pi_1 \rightarrow \pi_2$ in $\mathrm{Rep}^{\infty, F}(P)$ is said to be {\it strict} if $f(\pi_1)$ is closed in $\pi_2$ (cf. \cite[Definition 1.1.1]{Sc99}).

\begin{corollary}\label{cor cw abelian category}
\begin{enumerate}
\item[(1)] The categories $\mathcal{CW}_{k}(P)$ and $\mathcal{HC}_{k}(P)$ are abelian. 
\item[(2)] In $\mathcal{CW}_{k}(P)$, any monic morphism is injective and any epi morphism is surjective. 
\item[(3)] In $\mathcal{CW}_{k}(P)$, any morphism is strict.
\end{enumerate}
\end{corollary}

\begin{proof}
It is clear that $\mathcal{HC}_{k}(P)$ is abelian and so $\mathcal{CW}_{k}(P)$ is also abelian by Theorem \ref{thm equivalence of categories}. The assertion (2) follows from \cite[Proposition 7.12]{Ca89}, and the uniqueness of the Casselman-Wallach globalization. The assertion (3) follows from \cite[Lemma 7.15]{Ca89} and the uniqueness of the Casselman-Wallach globalization.
\end{proof}

\section{Adjointness of (generalized) parabolic induction and Jacquet functor} \label{adj para and jacq}

The abelian  categorical structure provides a nice control on the topology of Casselman-Wallach representations. In this section, we explicitly extract information for Casselman-Jacquet modules. Along the way, we show some results needed to prove the globalization in Section \ref{s globalization result}. 

\subsection{Parabolic inductions for Casselman-Wallach representations} \label{ss parabolic ind cw}

Let $\sigma$ be in $\mathcal{CW}_k(P)$. Define the unnormalized induction $\mathrm{Ind}_P^{G, \infty}\sigma$ to be the space of smooth functions $f: G \rightarrow \sigma$ satisfying
\begin{align} \label{def induced functions}
   f(pg)= p.(f(g)) \quad \mbox{for $p \in P$ and $g \in G$ } .
\end{align}
We shall equip $\mathrm{Ind}_P^{G, \infty}\sigma$ with the $C^{\infty}$-topology. The induced module $\mathrm{Ind}_P^{G, \infty}\sigma$ is a smooth Fr\'echet representation of $G$ (see  \cite[Lemma 10.1.1]{Wa92} \footnote{The primary concern in \cite[Section 10.1.1]{Wa92} is on $\mathcal{CW}_1(P)$, but the induced space is constructed from a compact model and so \cite[Lemma 10.1.1]{Wa92} is still valid in our setting.}). It follows from \cite[Proposition 4.1]{Ca89} that $\mathrm{Ind}_P^{G, \infty}: \mathcal{CW}_k(P)\rightarrow \mathcal{CW}(G)$ is a well-defined functor. Moreover, the functor is exact (see \cite[Proposition 2.2.7]{dCl91}). The way of the functor defined depends on $k$, but we shall suppress the index and it should be clear from the context.







\subsection{Parabolic inductions for Harish-Chandra modules}

We now also discuss the parabolic induction for Harish-Chandra modules. 

It is defined as follows. Let $\sigma$ be in $\mathcal{HC}_k(P)$ and let $\sigma^{\infty}$ be the associated globalization in Theorem \ref{thm equivalence of categories}. Define, following \cite[Chapter III, Section 2.3]{BW00} (cf. \cite[Proposition 11.47]{KV95}),
\[   I_0^{HC}(\sigma) =\mathrm{Hom}_{U(\mathfrak p)}(U(\mathfrak g), \sigma) =\left\{ f: U(\mathfrak g)\rightarrow \sigma : f(pg)=p.f(g) \mbox{ for all $p\in U(\mathfrak p)$, $g \in U(\mathfrak g)$} \right\} ,
\] 
where $U(\mathfrak g)$ (resp. $U(\mathfrak p)$) is the universal enveloping algebra of $\mathfrak g$ (resp. $\mathfrak p$).

The action is determined by $(X.f)(Y)=f(YX)$ (for $X\in \mathfrak g$ and $Y\in U(\mathfrak g)$, $f \in I_0^{HC}(\sigma)$).  Let $I_1^{HC}(\sigma)$ be the space of $\mathfrak k$-finite vectors $f$ in $I_0^{HC}(\sigma)$ such that $U(\mathfrak k).f$ is semisimple as $\mathfrak k$-module. Let $\widetilde{K}_P$ be the universal covering group of $K_P$ with the covering map denoted by $\mathrm{pr}$ and let $Z_P$ be the kernel of the covering map $\mathrm{pr}$. Since $I_1^{HC}(\sigma)$ is a semisimple $\mathfrak k$-module and $\widetilde{K}_P$ is simply connected, one lifts the $\mathfrak k$-action to $\widetilde{K}_P$-action. To add the compatibility of the $K_P$-action, we define $\mathrm{Ind}_{\mathfrak p, K_P}^{\mathfrak g, K}(\sigma)$ to be precisely the subspace of $I_1^{HC}(\sigma)$ consisting of functions satisfying: for $k \in K_P$ and $X \in U(\mathfrak g)$,
\[   (k.f)(X)=k.f(\mathrm{Ad}(k)(X)) .
\]
One sees that $Z_P$ acts by an identity on the above space, and so the $\widetilde{K}_P$-action descends to a $K_P$-action on the space $\sigma$.

 \begin{proposition} \label{prop gK induction} \cite[Chapter III, Proposition 2.4]{BW00} (cf. \cite[Proposition 11.47]{KV95})
For $\sigma$ in $\mathcal{HC}_k(P)$, there is a natural isomorphism $\mathrm{Ind}_{\mathfrak p, K_P}^{\mathfrak g,K}(\sigma) \cong (\mathrm{Ind}_P^{G, \infty}\sigma^{\infty})_K$.
\end{proposition}

\begin{proof}
This follows by the same argument in \cite[Chapter III, Proposition 2.4]{BW00}.
\end{proof}



\subsection{Jacquet functor for Casselman-Wallach representations}

\begin{definition} \label{def jacquet functor} \label{def partial cj cw}
For $k \in \mathbb Z_{\geq 1}$, define the {\it $k$-th Casselman-Jacquet functor} 
\[ \mathbf{CJ}^{\infty}_{P,k}:\mathcal{CW}(G)\rightarrow \mathcal{CW}_k(P), \quad  \mathbf{CJ}^{\infty}_{P,k}(\pi)=  \pi/\overline{\mathfrak n^k.\pi} .
\]
\end{definition}

The well-definedness of $\mathbf{CJ}^{\infty}_{P,k}$ follows from:

\begin{lemma} \label{lem jacquet casselman wallach rep} (see e.g. \cite[Page 86]{Wa92})
Let $\pi$ be in $\mathcal{CW}(G)$.  Then, for any $k \in \mathbb Z_{\geq 1}$, $\pi/\overline{\mathfrak n^k.\pi}$ is a Casselman-Wallach representation of $M_P$, and so is in $\mathcal{CW}_k(P)$.
\end{lemma}

We also have a quite standard Frobenius reciprocity (cf. \cite[Page 559]{Ca78}) and we omit the standard proof.

\begin{proposition} \label{prop frobenius smooth} 
Let $k \in \mathbb Z_{\geq 1}$. For any $\pi$ in $\mathcal{CW}(G)$ and any $\sigma$ in $\mathcal{CW}_{k}(P)$, there is a natural isomorphism:
\[  \mathrm{Hom}_G(\pi, \mathrm{Ind}^{G,\infty}_P\sigma) \cong \mathrm{Hom}_P(\mathbf{CJ}^{\infty}_{P,k}(\pi), \sigma) .
\]
\end{proposition}


\subsection{Jacquet functor for Harish-Chandra modules}

\begin{definition} \label{def partial cj hc}
For $k \in \mathbb Z_{\geq 1}$, define the $k$-th Casselman-Jacquet functor:
\[ \mathbf{CJ}^{HC}_{P,k}:\mathcal{HC}(G) \rightarrow \mathcal{HC}_k(P), \quad \mathbf{CJ}^{HC}_{P,k}(\pi)=\pi/(\mathfrak n^k.\pi) . \]
\end{definition}

We again have a Frobenius reciprocity, and the proof is also quite standard, see \cite[Theorem 4.9]{HS83} or \cite[Chapter III, Proposition 2.5]{BW00}.

\begin{proposition} \label{prop frobenius reciprocity hc} \cite{Ca78}  Let $k\in \mathbb Z_{\geq 1}$. For any $\pi$ in $\mathcal{HC}(G)$ and $\sigma$ in $\mathcal{HC}_{k}(P)$, there is a natural isomorphism:
\[   \mathrm{Hom}_{\mathfrak g,K}(\pi, \mathrm{Ind}_{\mathfrak p, K_P}^{\mathfrak g, K}\sigma) \cong \mathrm{Hom}_{\mathfrak p,K_P}(\mathbf{CJ}^{HC}_{P,k}(\pi), \sigma) .\]
\end{proposition}


\subsection{A topological consequence}

\begin{lemma} \label{lem cj right exact}
For any $k \in \mathbb Z_{k \geq 1}$, the functor $\mathbf{CJ}_{P,k}^{\infty}$ is right exact.
\end{lemma}

\begin{proof}

This follows from Proposition \ref{prop frobenius smooth}, which asserts that  $\mathbf{CJ}_{P,k}^{\infty}$ is left adjoint to $\mathrm{Ind}_P^{G,\infty}$, and the standard fact in homological algebra (for an abelian category) that any left adjoint functor is right exact.
\end{proof}

\begin{corollary} \label{cor surjective closure}

Let $\pi, \pi'$ be in $\mathcal{CW}(G)$ with a surjective morphism $f$ from $\pi$ to $\pi'$.  Then, for any $k \in \mathbb Z_{\geq 1}$,
 \[ f(\overline{\mathfrak n^k.\pi})=\overline{\mathfrak n^k.\pi'}.  \]
\end{corollary}

\begin{proof}
It is clear that $f(\overline{\mathfrak n^k.\pi})\subset \overline{\mathfrak n^k.\pi'}$. We consider the short exact sequence in $\mathcal{CW}(G)$:
\[  0 \rightarrow \omega \stackrel{\iota}{\rightarrow} \pi \stackrel{f}{\rightarrow} \pi' \rightarrow 0,
\]
where $\omega$ is the kernel of $f$. Then, by Lemma \ref{lem cj right exact}, this induces a right exact sequence in $\mathcal{CW}_k(P)$:
\[  \omega/\overline{\mathfrak n^k.\omega} \stackrel{\widetilde{\iota}}{\rightarrow} \pi/\overline{\mathfrak n^k.\pi} \stackrel{\widetilde{f}}{\rightarrow} \pi'/\overline{\mathfrak n^k.\pi'} \rightarrow 0 .
\]
Then, Corollary \ref{cor cw abelian category}(3) implies $\widetilde{\iota}(\omega/\overline{\mathfrak n^k.\omega})=\widetilde{f}^{-1}(0)$. This forces that the inclusion $f(\overline{\mathfrak n^k.\pi})\subset \overline{\mathfrak n^k.\pi'}$ is an equality.
\end{proof}

Corollary \ref{cor surjective closure} strengthens the surjectivity onto the Casselman-Jacquet quotient in Proposition \ref{prop surj cj quotient} (also see Corollary \ref{cor exact cw cj sub} below). This demonstrates a nicer topological behaviour for the Casselman-Wallach category. On the other hand, it is believed that $\mathfrak n^k.\pi$ is closed in $\pi$ while a general proof seems not to be known. Corollary \ref{cor surjective closure} also reduces the folklore conjecture to the principal series case through the (dual of the) Casselman embedding.

\section{Globalization of Casselman-Jacquet modules} \label{s globalization result}

In this section, we supply the globalization of the partial Casselman-Jacquet modules. Our strategy to show the globalization is to use the globalization of parabolic inductions and then transfer to Casselman-Jacquet functors from uniqueness of adjoint functors.

\begin{theorem} \label{prop cw jacquet functor}
Let $\pi$ be in $\mathcal{CW}(G)$. For any $k \in \mathbb Z_{\geq 1}$, $\mathbf{CJ}_{P,k}^{\infty}(\pi)$ is naturally isomorphic to the Casselman-Wallach globalization of $\mathbf{CJ}_{P,k}^{HC}(\pi_K)$.
\end{theorem}

\begin{proof}
By Proposition \ref{prop frobenius smooth}, 
\begin{align} \label{eqn isomorphism jacquet induction}
\mathrm{Hom}_G(\pi, \mathrm{Ind}_P^{G,\infty}\sigma)  \cong  \mathrm{Hom}_{P}(\mathbf{CJ}^{\infty}_{P,k}(\pi), \sigma) .
\end{align}
 On the other hand, we have: 
\begin{align*} \label{eqn isomorphisms jacquet hc}
\mathrm{Hom}_G(\pi, \mathrm{Ind}_P^{G, \infty}\sigma) & \cong \mathrm{Hom}_{(\mathfrak g, K)}(\pi_K, (\mathrm{Ind}_P^{G, \infty}\sigma)_K) \\                             & \cong \mathrm{Hom}_{(\mathfrak g, K)}(\pi_K, \mathrm{Ind}_{(\mathfrak p, K_P)}^{(\mathfrak g, K)}(\sigma_{K_P})) \\                                        & \cong \mathrm{Hom}_{(\mathfrak p, K_{P})}(\mathbf{CJ}^{HC}_{P,k}(\pi_K), \sigma_{K_P} ) \\
    & \cong \mathrm{Hom}_P((\mathbf{CJ}^{HC}_{P,k}(\pi_K))^{\infty}, \sigma)
\end{align*}
where the first and last isomorphisms follow from Theorem \ref{thm equivalence of categories}, the second isomorphism follows from Proposition \ref{prop gK induction}, and the third isomorphism follows from Proposition \ref{prop frobenius reciprocity hc}. 

The uniqueness of the adjoint functors in abelian categories (see Corollary \ref{cor cw abelian category}(1)) then implies that $\mathbf{CJ}^{\infty}_{P,k}(\pi)$ is naturally isomorphic to $(\mathbf{CJ}_{P,k}^{HC}(\pi_K))^{\infty}$. This shows the theorem.
\end{proof}

\begin{remark} \label{rmk explicit globalization}
We shall describe the natural isomorphism in Theorem \ref{prop cw jacquet functor}. To do so, we shall recall some generality for abelian categories. Let $\mathcal A$ and $\mathcal B$ be abelian categories. Let $F: \mathcal A \rightarrow \mathcal B$ be an additive functor. Suppose $G$ and $G'$ are functors left adjoint to $F$. Then we have the following sequence of natural isomorphisms:
\[ \mathrm{Hom}_{\mathcal B}(G(X), G(X)) \cong \mathrm{Hom}_{\mathcal A}(X, F\circ G(X)) \cong \mathrm{Hom}_{\mathcal B}(G'(X), G(X)) .
\]
One starts with the identity morphism $\mathrm{Id}_{G(X)}$ on the leftmost spot, and then one results with a natural morphism $\Phi_X: G'(X) \rightarrow G(X)$ in $\mathrm{Hom}_{\mathcal B}(G'(X), G(X))$. 

With above, one deduces that the natural isomorphism from $(\mathbf{CJ}^{HC}_{P,k}(\pi_K))^{\infty}$ to $\mathbf{CJ}^{\infty}_{P,k}(\pi)$ arises from the projection $\pi_K/(\mathfrak n^k.\pi_K) \rightarrow \pi/\overline{\mathfrak n^k.\pi}$.
\end{remark}

\section{Artin-Rees lemma for Casselman-Wallach representations} \label{s artin rees lemma}

In this section, we prove the exactness of the Casselman-Jacquet functors in the Casselman-Wallach category. We have to establish a version of the Artin-Rees lemma, and a technical difficulty, compared with the Harish-Chandra case, is lack of some finite-generation. Our strategy is to use Theorem \ref{prop cw jacquet functor} to pass from Harish-Chandra modules to Casselman-Wallach representations.





\subsection{Artin-Rees lemmas}

We first recall the Artin-Rees lemma for $(\mathfrak g, K)$-modules:

\begin{lemma} \label{lem artin rees lemma} (see  \cite{McC67, Ca78, CO78, SW82}) \label{lem ARL hc}
Let $\pi$ be in $\mathcal{HC}(G)$ and let $\sigma$ be a $(\mathfrak g, K)$-submodule of $\pi$. Then there exists a sufficiently large positive integer $k$ such that, for any $n >k$,
\[  \mathfrak n^{n-k}.(\mathfrak n^k.\pi \cap \sigma) = (\mathfrak n^n. \pi )\cap \sigma .
\]
As a consequence, $(\mathfrak n^n. \pi) \cap \sigma \subset \mathfrak n^{n-k}.\sigma$.
\end{lemma}

Since Lemma \ref{lem ARL hc} is quite important, we provide the proof for the convenience of the reader.

\noindent
\begin{proof} Let $B$ be a minimal parabolic subgroup contained in $P$. Let $\mathfrak n_B=\mathrm{Lie}(N_B)$. It follows from \cite[Corollary 1.2]{CO78} that $\pi$ is finitely generated as a $U(\mathfrak n_B)$-module.

Let $I=\mathfrak nU(\mathfrak n_B)$. It follows from \cite[Theorem 4.2]{McC67} (also see \cite[Theorem 2.1]{SW82}) that the Artin-Rees property of the ideal $I$ in $U(\mathfrak n_B)$ is satisfied, i.e. there exists a sufficiently large positive integer $k$ such that for any $n>k$,
\[  I^{n-k}.(I^k.\pi \cap \sigma) = (I^n.\pi)\cap \sigma .
\]
This can then be rewritten as $\mathfrak n^{n-k}.(\mathfrak n^k.\pi \cap \sigma) = (\mathfrak n^n.\pi)\cap \sigma$.
 \end{proof}







Recall that $\overline{\mathfrak n^n.\pi\cap \sigma} \subset \overline{\mathfrak n^n.\pi}\cap \overline{\sigma}$, and so the proof of Lemma \ref{lem AR lemma} does not simply follow from taking the closure in  Lemma \ref{lem artin rees lemma}.

\begin{lemma} \label{lem AR lemma} (Artin-Rees lemma for Casselman-Wallach category)
Let $\pi$ be in $\mathcal{CW}(G)$ and let $\sigma$ be a Casselman-Wallach subrepresentation of $\pi$. Then, there exists a sufficiently large positive integer $k$ such that for any $n>k$,
\[ \overline{\mathfrak n^n.\pi} \cap \sigma \subset \overline{\mathfrak n^{n-k}.\sigma}   .
\]
\end{lemma}

\begin{proof}
 The Casselman-Wallach equivalence of categories then provides an embedding for $(\mathfrak g, K)$-modules 
\[ \sigma_K \hookrightarrow \pi_K. \]
This then induces a $(\mathfrak p, K_P)$-morphism:
\[   \sigma_K/(\mathfrak n^n. \sigma_K) \rightarrow  \pi_K/(\mathfrak n^n. \pi_K) .
\]

 We consider the following commutative diagram such that the two horizontal lines are left-exact sequences:
\[  \xymatrix{ 0 \ar[r] & (\sigma_K\cap (\mathfrak n^n.\pi_K)) /(\mathfrak n^n. \sigma_K) \ar[r] \ar[d] & \sigma_K/(\mathfrak n^n. \sigma_K) \ar[r] \ar[d] &  \pi_K/(\mathfrak n^n. \pi_K) \ar[d] \\ 0 \ar[r] &
(\sigma \cap \overline{\mathfrak n^n.\pi})/(\overline{\mathfrak n^n. \sigma}) \ar[r]  & \sigma/\overline{\mathfrak n^n.\sigma} \ar[r] & \pi/\overline{\mathfrak n^n.\pi} }  .
\]
The last two vertical maps are induced from the embedding $\sigma_K$ to $\sigma$, and $\pi_K$ to $\pi$. It follows from Proposition \ref{prop cw jacquet functor} (see Remark \ref{rmk explicit globalization}) that the two maps are injective maps determining the Casselman-Wallach globalization. Since the Casselman-Wallach globalization functor is exact, the image of the first map is dense. In other words, the projection of $(\sigma_K \cap (\mathfrak n^n.\pi_K))/(\overline{\mathfrak n^n.\sigma})$ into $(\sigma \cap \overline{\mathfrak n^n.\pi})/(\overline{\mathfrak n^n.\sigma})$ is dense.

We choose the index $k$ as the one in the Artin-Rees lemma for Harish-Chandra modules in Lemma \ref{lem artin rees lemma} and consider the morphism induced from the embedding $\sigma \cap \overline{\mathfrak n^n.\pi} \hookrightarrow \sigma$:
\[ F: (\sigma \cap \overline{\mathfrak n^n.\pi})/\overline{\mathfrak n^n.\sigma} \rightarrow \sigma/(\overline{\mathfrak n^{n-k}.\sigma}) .
\]
By the Artin-Rees lemma for Harish-Chandra modules (see Lemma \ref{lem artin rees lemma}), we then have a dense set in $(\sigma \cap \overline{\mathfrak n^n.\pi})/\overline{\mathfrak n^n.\sigma} $ mapping to zero under the above map $F$. Hence, by continuity, the projection is zero. This implies that
\[  \sigma \cap \overline{\mathfrak n^n.\pi} \subset \overline{\mathfrak n^{n-k}.\sigma} 
\]
as desired.
\end{proof}

\subsection{Exactness of the Casselman-Jacquet functors}

Recall that $\mathbf{CJ}^{\infty}_P$ is defined in Section \ref{ss two functors}. The following proof of Theorem \ref{thm exact cj functors cw} is inspired by the exactness of the localization functor in commutative algebra, and a main ingredient is a version of Artin-Rees lemma (Lemma \ref{lem AR lemma}).

A short exact sequence $0 \rightarrow \pi_1 \stackrel{f_1}{\rightarrow} \pi_2 \stackrel{f_2}{\rightarrow} \pi_3\rightarrow 0$ in $\mathrm{Rep}^{\infty,F}(P)$ is said to be strict if all morphisms in the sequence are strict. Hence, $f_1$ is injective and $f_2$ is surjective, and so $\pi_2/\pi_1 \cong \pi_3$.

\begin{theorem} \label{thm exact cj functors cw}
Let $\iota: \mathcal{CW}(G)\rightarrow \mathrm{Rep}^{\infty,F}(G)$ be the natural embedding. Then the Casselman-Jacquet quotient functor $\mathbf{CJ}^{\infty}_P\circ \iota$ sends a short exact sequence to a strict short exact sequence i.e. exact in the sense of \cite[Definition 1.1.18]{Sc99}.
\end{theorem}

\begin{proof}

 Let $0 \rightarrow \pi_1 \rightarrow \pi_2 \rightarrow \pi_3 \rightarrow 0$ be an exact sequence of Casselman-Wallach representations of $G$. The surjection part follows from a more general result in Proposition \ref{prop surj cj quotient functor}.


We now sketch the proof for the middle exactness and the injectivity for the convenience of the reader. We have the following short exact sequence:
\[   0 \rightarrow \varprojlim~ \pi_1/(\pi_1\cap \overline{\mathfrak n^k.\pi_2}) \rightarrow \varprojlim~ \pi_2/\overline{\mathfrak n^k. \pi_2} \rightarrow \varprojlim~\pi_3/\overline{\mathfrak n^k. \pi_3} \rightarrow 0,
\]
where the above surjectivity follows from Proposition  \ref{prop surj cj quotient functor}, and the injectivity and the middle exactness follows from the definition of the induced maps in the inverse limit. Note that the natural map from $\mathbf{CJ}^{\infty}_P(\pi_1)$ to $\mathbf{CJ}^{\infty}_P(\pi_2)$ factors through the natural projection $\mathrm{pr}$ from $\varprojlim~ \pi_1/\overline{\mathfrak n^k.\pi_1}$ to $\varprojlim~ \pi_1/(\pi_1\cap \overline{\mathfrak n^k.\pi_2})$. Thus, it suffices to show that $\mathrm{pr}$ is an isomorphism. This follows from Lemma \ref{lem AR lemma} that one can find an increasing sequence $k_1<k_2<k_3< \ldots$ to define a natural map 
\[  \varprojlim_i \pi_1/(\pi_1 \cap \overline{\mathfrak n^{k_i}.\pi_2}) \rightarrow \varprojlim_k \pi_1/\overline{\mathfrak n^k.\pi_1}  .
\]
With the fact that $\varprojlim_i \pi_1/(\pi_1\cap \overline{\mathfrak n^{k_i}.\pi_2}) \cong \varprojlim_k \pi_1/(\pi_1\cap \overline{\mathfrak n^k.\pi_2})$, one defines the inverse of $\mathrm{pr}$ and shows that $\mathrm{pr}$ is an isomorphism.
\end{proof}

\begin{corollary} \label{cor exact cw cj sub}
The Casselman-Jacquet submodule functor $\mathbf{CJ}_{s,P}^{\infty}$ restricted to the Casselman-Wallach category $\mathcal{CW}(G)$ sends a short exact sequence to a strict short exact sequence.
\end{corollary}

\begin{proof}
By Proposition \ref{prop surj cj quotient}, we have a short exact sequence: 
\[   0 \rightarrow \mathbf{CJ}^{\infty}_{s,P}(\pi) \rightarrow \pi \rightarrow \mathbf{CJ}_P^{\infty}(\pi) \rightarrow 0 .\]
The exactness of $\mathbf{CJ}^{\infty}_{s,P}(\pi)$ then follows from the exactness of $\mathbf{CJ}_P^{\infty}(\pi)$ and a standard argument of tracing commutative diagrams.
\end{proof}

\part{Bernstein-Zelevinsky filtrations} \label{part bz filtration}

We begin with notation that will be used throughout the remaining sections for real general linear groups. Let $\mathbb K=\mathbb R$ or $\mathbb C$. Let $G_n=\mathrm{GL}_n(\mathbb K)$. Let $B_n$ be the subgroup of upper triangular matrices in $\mathrm{GL}_n(\mathbb K)$.  For $t \in \mathbb K$, let $|t|$ be the absolute value of $t$.

\section{Schwartz and mirabolic inductions} \label{s mirabolic subgp}

In this section, we shall provide an alternate self-contained definition of Schwartz inductions for mirabolic subgroups, which is more algebraic in nature. The reader who is familiar with \cite{dCl91} can adapt the definition of Schwartz induction in \cite[2.1.2 D\'EFINITION]{dCl91} and skip some details in the definition. The advantage of our presentation is to realize the mirabolic inductions as a concrete subspace of some functions.

\subsection{Mirabolic subgroups} \label{ss mirabolic subgp}

The mirabolic subgroup $M_{n}$ of $\mathrm{GL}_{n}(\mathbb K)$ is defined as:
\[ M_n= \left\{   \begin{pmatrix}  g & v \\ 0 & 1 \end{pmatrix} : g \in \mathrm{GL}_{n-1}(\mathbb K) \mbox{ and } v \in \mathbb K^{n-1} \right\} .
\]
Let 
\[  V_{n-1}= \left\{ \begin{pmatrix} I_{n-1} & v \\  0 & 1 \end{pmatrix} : v \in \mathbb K^{n-1} \right\} , \quad \mathfrak v=\mathfrak v_{n-1}=\mathrm{Lie}(V_{n-1}) .
\]
We shall regard $M_n$ and $G_n$ as subgroups of $M_{n+1}$ via the embedding $m \mapsto \begin{pmatrix}  m & \\ & 1 \end{pmatrix}$. 

\subsection{Compact realizations} \label{ss compact realize}

Define
\[   K_n = \left\{ \begin{array}{ll} O(n) &  \mbox{ if $\mathbb K=\mathbb R$} \\
      U(n) &   \mbox{ if $\mathbb K=\mathbb C$ } \end{array} \right.  ,
\]
the orthogonal and unitary subgroups in $G_n$ respectively. Define $\mathrm{Ind}_{K_{n-1}}^{K_n}\sigma$ to be the space of smooth functions from $K_n$ to $\sigma$ such that for all $k \in K_{n-1}$, $f(kg)=k.f(g)$ for all $g \in K_n$. For $\sigma$ in $\mathrm{Rep}^{\infty,F}(K_{n-1})$, define $\mathrm{SInd}_{K_{n-1}}^{K_n \times \mathbb R_{>0}}\sigma:=\mathrm{SInd}_1^{\mathbb R_{>0}}(\mathrm{Ind}_{K_{n-1}}^{K_n}\sigma)$ to be the space of smooth functions from $K_n \times \mathbb R_{>0}$ to $\sigma$ satisfying: for any $k \in K_{n-1}$ and $(k',r) \in K_n \times \mathbb R_{>0}$,
\[  f((k,1)(k',r))=k.f(k',r)
\]
and furthermore $r^s\frac{d^{l}}{dr^l}f(k,r)$, $s\in \mathbb{Z}_{\geq 0}, l\in \mathbb{Z}_{\geq 0}$, vanish for $r$ tending to $0$ or infinity.




\subsection{Mirabolic inductions} \label{ss mir induction}

Note that we can embed $K_{n-1}$ to $M_n$ via $k \mapsto \mathrm{diag}(k,1)$, and $K_n\cap M_n=K_{n-1}$. With this, one has that, as Nash manifolds,
\[  (M_{n}V_n) \backslash M_{n+1} \cong M_n\setminus G_n \cong K_{n-1}\backslash K_n \times \mathbb R_{>0} 
\]
and so the space $\mathrm{SInd}_{K_{n-1}}^{K_n \times \mathbb R_{>0}}\sigma$ defined in Section \ref{ss compact realize} will suitably define the mirabolic induction.
We now explicitly define the functor:
\[   \mathrm{SInd}_{M_nV_n}^{M_{n+1}}: \mathrm{Rep}^{\infty, F}(M_nV_n) \rightarrow \mathrm{Rep}^{\infty, F}(M_{n+1}) ,
\]
For $\pi \in \mathrm{Rep}^{\infty,F}(M_nV_n)$,  the underlying space of $\mathrm{SInd}_{M_nV_n}^{M_{n+1}}\pi$ is $\mathrm{SInd}_{K_{n-1} }^{K_n \times \mathbb{R}_{>0}} (\pi|_{K_{n-1}})$. We have to describe the $M_{n+1}$-action on $\mathrm{SInd}_{M_nV_n}^{M_{n+1}}\pi$.

We have the following continuous surjection:
\begin{align} \label{eqn gl decompose}
  V_n\times G_n  \simeq  V_n \times (B_n\cap M_n) \times  K_n \times \mathbb R_{>0} \twoheadrightarrow  M_{n+1} ,
\end{align}
via the map 
\[ V_n \times (B_n\cap M_n)  \times  K_n \times \mathbb R_{>0}\ni (w, b, k, r) \mapsto w\cdot \mathrm{diag}(b,1) \cdot \mathrm{diag}(k,1) \cdot \mathrm{diag}(r,\ldots, r, 1) .
\]

For each $f \in \mathrm{SInd}_{K_{n-1} }^{K_n \times \mathbb R_{>0}}\pi$, we define $\widetilde{f}: M_{n+1} \rightarrow \pi$ as follows. For an element $g \in M_{n+1}$, we express according to (\ref{eqn gl decompose})
\[  g=w\begin{pmatrix} b & \\ & 1\end{pmatrix} \begin{pmatrix} k & \\ & 1\end{pmatrix} \mathrm{diag}(r,\ldots, r,1)
\]
with $w\in V_n$, $k \in K_n$, $r \in \mathbb R_{>0}$, $b \in B_n \cap M_n$. We define:
\[   \widetilde{f}(g) =w \begin{pmatrix} b &  \\ & 1\end{pmatrix}.(f(k,r)) .
\]
Such the assignment is injective, and we denote the image by $\mathcal M$. The $M_{n+1}$-action on $\widetilde{f} \in \mathcal M$ is via the right translation, i.e. for $m, g \in M_{n+1}$, $(m.\widetilde{f})(g)=\widetilde{f}(gm)$. Note that, for $f \in \mathcal M$ and $m \in M_{n+1}$, $m.\widetilde{f}$ is still in $\mathcal M$, and this defines a $M_{n+1}$-action on $\mathrm{SInd}_{M_nV_n}^{M_{n+1}}\pi$.

The space $\mathrm{SInd}_{M_nV_n}^{M_{n+1}} \pi$ coincides with the {\it Schwartz induction} in the sense of \cite{dCl91}.

\begin{example} \label{ex mir M2}
We consider $\mathbb{K}=\mathbb{R}$ and the mirabolic subgroup $M_2$. Let $\pi$ be a one dimensional $V_1$-representation such that $\begin{pmatrix} 1 & a \\ & 1 \end{pmatrix} \in V_1$ acts on $\pi$ via multiplying the scalar $e^{-2\pi\sqrt{-1}a}$. Then $\mathrm{SInd}_{M_1V_1}^{M_2}\pi=\mathrm{SInd}_{V_1}^{M_2}\pi$ is isomorphic to $\mathcal S(\mathbb R^{\times})$ in Example \ref{ex schwartz function M2}.
\end{example}

\section{Links the Casselman-Jacquet submodule functor to mirabolic inductions} \label{s descrip bz layer}

We shall usually identify $\mathfrak v_n$ with $\mathbb K^n$ by the map $\begin{pmatrix} 0_n & v \\ & 0 \end{pmatrix} \mapsto v$. Let $e_n$ denote the vector $(0,\dots,0,1)^T\in \mathbb K^n$. Let $\phi:  V_n \rightarrow \mathbb C$ be the character given by: for $v \in \mathfrak v_n$,
\[   \phi(\mathrm{exp}(v)) =e^{-2\sqrt{-1}\pi\mathrm{Re}( \langle e_n^T, v\rangle)}.
\]
Let $\pi \in \mathrm{Rep}^{\infty,F}(M_{n+1})$. We can then form the $M_{n+1}$-representation $\mathcal S(\mathfrak n)\widehat{\otimes}\pi$ as in Section \ref{ss first intersect prop}. We shall consider the $M_{n+1}$-subspace 
\[ ( \bigcap_{k =0}^{\infty} \mathcal C_k(\mathfrak v_n) )\widehat{\otimes} \pi  \subset \mathcal S(\mathfrak v_n)\widehat{\otimes}\pi .
\]
Our first goal is to transform the above subspace into mirabolic inductions via the following three identifications. Example \ref{ex mir M2} illustrates the identifications for $n=1$.

\subsection{First identification} \label{ss induction realize 1}

The Fourier transform in Section \ref{ss fourier transform} provides an isomorphism of $M_{n+1}$-representations:
\[ \Theta_1: (\bigcap_{k =0}^{\infty} \mathcal C_k(\mathfrak v_n)) \widehat{\otimes} \pi \stackrel{\sim}{\rightarrow} \mathcal S(\mathfrak{v}_n^*\setminus \left\{ 0 \right\}) \widehat{\otimes} \pi . \]

We shall identify $\mathfrak v_n^*$ with the space of row vectors over $\mathbb K$ with $n$-entries, and fix the pairing $\mathfrak v_n^*\times \mathfrak v_n$ given by: $\langle y,x\rangle =yx$, the matrix multiplication between a row vector and a column vector. The $M_{n+1}$-structure on $\mathcal S(\mathfrak v_n^*)\widehat{\otimes} \pi$ can be described explicitly. The subgroup $V_n$ acts on $f\in \mathcal S(\mathfrak{v}_n^*\setminus \left\{ 0 \right\}) $ as follows: for $y \in \mathfrak v_n^*\setminus \left\{ 0 \right\}$, and $v\in \mathfrak v_n$,
\[ \left(\mathrm{exp}(v).f\right)(y)= \left(\begin{pmatrix} I_n & v \\ & 1\end{pmatrix}.f\right)(y) =  e^{-2\pi\sqrt{-1}\mathrm{Re}\langle y,v \rangle}f(y) ,
\]
and the action of $V_n$ on $\pi$ is trivial. The subgroup $G_n$ acts on  $\mathcal S(\mathfrak{v}_n^*\setminus \left\{ 0 \right\}) $ by: $(g.f)(y)=\mathrm{det}(g) \cdot f(yg)$, and acts on $\pi$ from the $M_{n+1}$-action on $\pi$.

Recall that the action map $a:\mathcal S(\mathfrak v_n)\widehat{\otimes}\pi \rightarrow \pi$ in (\ref{eqn n action}) is given by:
\[     f\otimes v \in  (\bigcap_{k=0}^{\infty}\mathcal C_k(\mathfrak v_n))\widehat{\otimes} \pi \stackrel{a}{\mapsto} \int_{\mathfrak v_n} f(y) \mathrm{exp}(y).v ~dy  .
\]
For $f \in \mathcal S(\mathfrak v_n^*\setminus \left\{ 0 \right\})$, we now have:
\begin{align} \label{eqn first identify}  f \otimes v \stackrel{\Theta_1^{-1}}{\mapsto} \widehat{f}\otimes v \stackrel{a}{\mapsto}  \int_{\mathfrak v_n }  \int_{\mathfrak{v}_n^*\setminus \left\{0\right\}} e^{2{\sqrt{-1} \pi \mathrm{Re}\langle x, y \rangle}} f(x) ~dx ~\mathrm{exp}(y).v ~dy  .
\end{align}

\subsection{Second identification} \label{ss induction realize 2}

We have isomorphisms of manifolds:
\[ (M_nV_n)\backslash M_{n+1} \cong M_n\setminus G_n \simeq \mathfrak{v}_n^*\setminus \left\{ 0 \right\}\simeq (\mathbb K^n)^*\setminus \left\{ 0 \right\} \]
and it is explicitly determined by:
for $m \in G_n \subset M_{n+1}$,
    \[ m\mapsto  e_n^T\cdot m.
    \]
This induces the following identification, which can be directly checked from the explicit description in Section \ref{ss mir induction}:
\[ \Theta_2: \mathcal S(\mathfrak{v}_n^*\setminus \left\{ 0 \right\}) \widehat{\otimes} \pi \stackrel{\sim}{\rightarrow} \mathrm{SInd}_{M_nV_n}^{M_{n+1}}(\mathbb C_{\mathrm{triv}}\boxtimes \phi)\widehat{\otimes} \pi . \]
For $h \in \mathrm{SInd}_{M_nV_n}^{M_{n+1}}(\mathbb C_{\mathrm{triv}})$ and $v\in \pi$, we can rewrite the map in the first identification as:
\[ h \otimes v  \stackrel{a \circ \Theta_1^{-1}\circ  \Theta_2^{-1}}{\longmapsto}   \int_{\mathfrak v_n}  \left(\int_{(M_nV_n)\setminus M_{n+1}} e^{2\sqrt{-1}\pi \mathrm{Re}\langle e_n^T\cdot m, y \rangle} h(m) ~dm \right)~ \mathrm{exp}(y).v~dy  ,
\]
where $dm$ is the measure on $(M_nV_n) \backslash M_{n+1}$ from the change of variables from $dx$ in (\ref{eqn first identify}).

\subsection{Third identification} \label{ss induction realize 3}
The last isomorphism (cf. \cite[Theorem 44.1]{Tr06}) 
\[ \Theta_3: \mathrm{SInd}_{M_nV_n}^{M_{n+1}}(\mathbb C_{\mathrm{triv}}\boxtimes \phi)\widehat{\otimes} \pi \stackrel{\sim}{\rightarrow}  \mathrm{SInd}_{M_nV_n}^{M_{n+1}}(\pi|_{M_n}\boxtimes \phi) \]
is determined by:
    \[    f\otimes v \mapsto (m \in G_n \mapsto f(m) m.v) .    \]
Let $\Theta=\Theta_3\circ \Theta_2\circ \Theta_1$. For $h \in \mathrm{SInd}_{M_nV_n}^{M_{n+1}}\pi$, we finally have:
\begin{align*}
     h &\stackrel{a\circ \Theta^{-1}}{\longmapsto} \int_{\mathfrak v_n} \int_{M_nV_n \backslash M_{n+1}} e^{2\sqrt{-1}\pi \mathrm{Re}\langle e_n^T\cdot m, y \rangle}(\mathrm{exp}(y)m^{-1}). h(m)~ dm~  dy \\
      &= \int_{\mathfrak v_n} \int_{M_nV_n \backslash M_{n+1}} e^{2\sqrt{-1}\pi \mathrm{Re} \langle e_n^T\cdot m, y \rangle}(m^{-1}\cdot \mathrm{exp}(my)). h(m)~ dm~  dy,
\end{align*}
where the equation follows from $m\cdot\mathrm{exp}(y)\cdot m^{-1}=\mathrm{exp}(my)$.

\subsection{The functor $\Phi^+$}\label{def_phi+}

The action of $M_n$ on $\phi$ is given by $(m.\phi)(v)=\phi(m^{-1}v)$. The stabilizer subgroup of $\phi$ in $G_n$ is $M_n$ (see Section \ref{ss mirabolic subgp}).
 Define 
 \[ \Phi^+: \mathrm{Rep}^{\infty, F}(M_n) \rightarrow \mathrm{Rep}^{\infty, F}(M_{n+1}) \]
 given by:
\begin{equation} \label{eqn phi+ schwartz induction}
\Phi^+(\pi):= \mathrm{SInd}_{M_nV_n}^{M_{n+1}} \delta^{1/2}(\pi\boxtimes \phi) ,
 \end{equation}
where $\delta$ is the character of $M_{n}V_n$ given by $g \mapsto |\mathrm{det}(g)|$.

\begin{lemma} \label{lem action map kernel}
Let $\pi$ be in $\mathrm{Rep}^{\infty, F}(M_{n+1})$. Let
\[ \kappa = \overline{\mathfrak v_n.(\pi\otimes \phi^{-1})}  \subset \pi. \]
Then $\Phi^+(\kappa|_{M_n})$ lies in the kernel of the map 
\[a \circ \Theta^{-1}:  \Phi^+(\pi|_{M_n})=\mathrm{SInd}_{M_nV_n}^{M_{n+1}}(\pi|_{M_n}\boxtimes \phi) \rightarrow \pi  .\]
\end{lemma}

\begin{proof}
Let $h$ be in $\Phi^+(\kappa)$. We have to show $h \in \mathrm{ker}(a\circ \Theta^{-1})$.

We fix a compactly-supported open cover $\left\{ U_{\alpha}\right\}_{\alpha}$ on $(M_nV_n)\backslash M_{n+1}$ such that a smooth local section from $U_{\alpha}$ to $M_{n+1}$ exists on each $U_{\alpha}$. 
We fix a continuous section $\zeta=\zeta_{\alpha}$ from $U_{\alpha}$ to $M_{n+1}$ for each index $\alpha$. Using partition of unity on the cover $\left\{ U_{\alpha}\right\}$, we can then reduce to consider $h$ with $\mathrm{supp}(h)\subset U_{\alpha^*}$ for some $\alpha^*$. Further taking linearity and completion, it suffices  to consider $h$ satisfying that for any $m \in U_{\alpha^*}$,
\[    h(\zeta(m))= \widetilde{h}(\zeta(m))\cdot ( \mathrm{exp}(v_0).q_0 -\phi(\mathrm{exp}(v_0))q_0) 
\]
for a fixed $q_0 \in \pi$ and fixed $v_0 \in \mathfrak v_n$ and $\widetilde{h} $ is a complex-valued smooth functions on $M_{n+1}$ with $\mathrm{supp}(\widetilde{h})\subset U_{\alpha^*}$.

To this end, we first compute 
\[ \int_{\mathfrak v_n} \int_{M_nV_n \backslash M_{n+1}} e^{2\sqrt{-1}\pi\mathrm{Re}\langle e_n^T\cdot \zeta(m), y\rangle} \zeta(m)^{-1}\cdot( \mathrm{exp}(\zeta(m)y)\mathrm{exp}(v_0)). \widetilde{h}(m)q_0 ~dm ~dy 
\]
Then by combining $\zeta(m)y$ with $v_0$ (in $\mathfrak v_n$) and then changing variables, we have the expression
\[ e^{-2\sqrt{-1}\pi \mathrm{Re}\langle e_n^T, v_0 \rangle} \int_{\mathfrak v_n} \int_{M_nV_n \backslash M_{n+1}} e^{2\sqrt{-1}\pi \mathrm{Re}\langle e_n^T\cdot \zeta(m), y\rangle} \zeta(m)^{-1}(\mathrm{exp}(\zeta(m)y)).\widetilde{h}(\zeta(m))q_0 ~dm ~dy 
\]
and so is equal to 
\[ \phi(\mathrm{exp}(v_0))  \int_{\mathfrak v_n} \int_{M_nV_n\backslash M_{n+1}} e^{2\sqrt{-1}\pi \mathrm{Re}\langle e_n^T\cdot \zeta(m), y\rangle}  \zeta(m)^{-1}\cdot (\mathrm{exp}(\zeta(m)y)).\widetilde{h}(\zeta(m))q_0 ~dm ~dy  .
\]
Now, differentiation is applied to obtain $\Phi^+(\kappa)$ is in $\mathrm{ker}(a\circ \Theta^{-1})$.
\end{proof}



\section{Imprimitivity Theorem} \label{s imprimitivity thm}

\subsection{Imprimitivity theorem}

In order to further describe the Schwartz induction in Section \ref{s descrip bz layer}, we need an imprimitivity theorem. We first explain the setup.

Let $\sigma$ be in $\mathrm{Rep}^{\infty, F}(M_{n})$. Note that $\mathrm{SInd}_{M_n}^{G_n}$ can be defined in the same way as $\mathrm{SInd}_{M_nV_n}^{M_{n+1}}$. For $f_1 \in \mathrm{SInd}_{M_n}^{G_n}\mathbb C_{\mathrm{triv}}$ and  $f_2 \in \mathrm{SInd}_{M_n}^{G_n}\sigma$, we can regard the action of $f_1$ on $f_2$ via the multiplication $f_1f_2$ to be an element in $\mathrm{SInd}_{M_n}^{G_n}\sigma$. The space $\mathrm{SInd}_{M_n}^{G_n}\sigma$ also possess the $G_n$-action by right translation.

\begin{theorem} \cite[2.5.8. TH\'EOR\`ME(iii)]{dCl91} \label{thm imprimitivity}
\begin{enumerate}
\item Let $\sigma$ be in $\mathrm{Rep}^{\infty, F}(M_n)$. Let $\lambda$ be any closed subspace of $\mathrm{SInd}_{M_n}^{G_n}\sigma$ invariant under the action of $\mathrm{SInd}^{G_n}_{M_n}\mathbb C_{\mathrm{triv}}$ and the right translation $G_n$. Then $\lambda$ is isomorphic to $\mathrm{SInd}_{M_n}^{G_n}\sigma'$ for some smooth closed $M_n$-subrepresentation $\sigma'$ of $\sigma$.
\item Let $\sigma$ be an irreducible representation in $\mathrm{Rep}^{\infty, F}(M_{n})$. Then the only closed subspace of $\mathrm{SInd}_{M_n}^{G_n}\sigma$ invariant under the action of $\mathrm{SInd}_{M_n}^{G_n}\mathbb C_{\mathrm{triv}}$ and the right translation of $G_n$ is either $0$ or the whole space.
\end{enumerate}
\end{theorem}

We provide a heuristic proof of the imprimitivity theorem on a special case, which provides a basic idea on the above theorem. One considers $\sigma$ to be the trivial representation. In such case, we consider 
\[   M_n\backslash G_n\cong  (\mathbb{K}^n)^* \setminus \left\{ 0 \right\} ,
\]
where the isomorphism is given by  $g \mapsto e_n^T\cdot g$.

 Hence, we can view $\mathrm{SInd}_{M_n}^{G_n}\mathbb C_{\mathrm{triv}}$ as the space $\mathcal S((\mathbb{K}^n)^*\setminus \left\{ 0 \right\})$ of Schwartz functions on $(\mathbb{K}^n)^*\setminus \left\{ 0 \right\}$. One now proves that for any maximal ideal $I$ of $\mathcal S((\mathbb{K}^n)^*\setminus \left\{ 0 \right\})$,  
\[   I=\left\{ f \in \mathcal S((\mathbb{K}^n)^*\setminus \left\{ 0 \right\}): f(x_0)=0 \right\}
\]
for some $x_0$ in $(\mathbb{K}^n)^*\setminus \left\{ 0 \right\}$. Then one applies the transitive action $G_n$ on $(\mathbb{K}^n)^*\setminus \left\{ 0 \right\}$, and then one has that any proper subspace satisfying the required invariance properties must lie in 
\[ \bigcap_{x_0\in (\mathbb{K}^n)^*\setminus \left\{ 0 \right\}} \left\{ f \in \mathcal S((\mathbb{K}^n)^*\setminus \left\{0\right\}): f(x_0)=0 \right\} 
\]
and hence must be zero, as desired.

\subsection{Alternate form}

\begin{corollary} \label{cor alt form imprimitivity}
Let $\pi \in \mathrm{Rep}^{\infty, F}(M_n)$. Then any closed $M_{n+1}$-subrepresentation of $\Phi^+(\pi)$ takes the form $\Phi^+(\sigma)$ for some closed $M_n$-subrepresentation $\sigma$ of $\pi$.
\end{corollary}

\begin{proof}
From the $V_n$-action on $\Phi^+(\sigma)$, the composition of the Fourier transform and the corresponding action map of $V_n$ for $\mathcal S(\mathfrak v_n^*\setminus \left\{ 0 \right\})$ is given by: for $f \in \mathcal S(\mathfrak v_n^*\setminus \left\{ 0 \right\})$ and $h \in \Phi^+(\sigma)$, 
\[  f\otimes v \mapsto \widehat{f} v \quad \mbox{(multiplication of two functions)} .
\]
Now, the corollary follows from Theorem \ref{thm imprimitivity}. The details boil down to the identifications in Section \ref{s descrip bz layer} (one takes $\pi$ to be trivial), and so we omit that.
\end{proof}

\section{Hausdorffness of the functor $\Phi^-$} \label{s closed bz layer}

\subsection{The functor $\Phi^-$} \label{label character twist}

Define the functor 
\[  \Phi^-: \mathrm{Rep}^{\infty, F}(M_{n+1}) \rightarrow \mathrm{Rep}^{\infty, F}(M_{n})
\]
given by:
\[  \Phi^-(\pi) := \delta^{-1/2}\cdot \pi/(\mathfrak v. (\pi\otimes \phi^{-1})),
\]
where $\delta$ is the character of $M_{n}$ given by $m\mapsto |\mathrm{det}(m)|$.

We shall prove the Hausdorffness of $\Phi^-(\pi)$. A main idea is to apply for the du Cloux imprimitivity theorem so that we have rather explicit description on how the part contributing $\Phi^-$ looks like. The computation for the corresponding $\Phi^-$ is done in Lemma \ref{lem phi minus map}. Going back to compute $\Phi^-(\pi)$ is slightly technical, see Lemma \ref{lem kernel schwartz}. We also remark that \cite{AGS15a, AGS15b,WZ25+} establish Hausdorffness in Proposition \ref{prop bz description} for Casselman-Wallach representations by different methods.

\subsection{Evaluation map}

\begin{lemma} \label{lem phi minus map}
Let $\omega$ be in $\mathrm{Rep}^{\infty, F}(M_n)$. Then $\Phi^-\circ \Phi^+(\omega) \cong \omega$ as representations in $\mathrm{Rep}^{\infty, F}(M_{n})$. Moreover, using the realization in (\ref{eqn phi+ schwartz induction}), the isomorphism is given by evaluating at the identity in the space $\mathrm{SInd}_{M_nV_n}^{M_{n+1}}(\omega \boxtimes \phi)$.
\end{lemma}

\begin{proof}
We have to compute the action $\Phi^-$ on $\Phi^+(\omega) \cong \mathrm{SInd}_{M_nV_n}^{M_{n+1}}(\omega \boxtimes \phi)$. We shall extend $\omega$ to a trivial $V_n$-representation. Note that as $V_n$-representations, we have:
\[  \mathrm{SInd}_{M_nV_n}^{M_{n+1}} (\omega \boxtimes \phi) \cong \mathrm{SInd}_{M_nV_n}^{M_{n+1}}(\mathbb C_{\mathrm{triv}}\boxtimes \phi) \widehat{\otimes} \omega \cong \mathcal S(\mathfrak v^*\setminus \left\{ 0 \right\})\widehat{\otimes} \omega \cong \mathcal S(\mathfrak v^*\setminus \left\{ 0 \right\}, \omega) ,
\]
where the first isomorphism is similar to the one $\Theta_3$ in Section \ref{ss induction realize 3}; the second isomorphism is similar to the one $\Theta_2$ in Section \ref{ss induction realize 2}; and the last isomorphism is determined by $f\otimes v\mapsto (m \mapsto f(m)v$), where $f \in \mathcal S(\mathfrak v^* \setminus \left\{ 0 \right\})$ and $v \in \omega$.

Then it is direct to show that $\mathfrak v.\left(\mathcal S(\mathfrak v^*\setminus \left\{ 0 \right\}, \omega)\otimes \phi^{-1}\right)$ is the space of functions vanishing at $e_n^T$. Translating under above isomorphisms, we obtain $\mathfrak v.  \mathrm{SInd}_{M_nV_n}^{M_{n+1}} (\omega \boxtimes \phi)$ to be the space of all functions that vanish at the identity. This implies the lemma. 
\end{proof}

\subsection{Hausdorffness}

\begin{lemma} \label{lem kernel schwartz} (cf. \cite[2.5.7 LEMME]{dCl91})
 Let $\pi$ be in $\mathrm{Rep}^{\infty, F}(M_{n+1})$. We use the notation in Section \ref{s descrip bz layer}. Let $\kappa'$ be the maximal closed $M_n$-subspace of $\pi|_{M_n}$ such that $a\circ \Theta^{-1}(\mathrm{SInd}_{M_nV_n}^{M_{n+1}}(\kappa'\boxtimes \phi)=0$.  Then 
\begin{enumerate}
\item $\kappa' = \mathfrak v.(\pi \otimes \phi^{-1})=\overline{\mathfrak v.(\pi\otimes \phi^{-1})}$;
\item $\Phi^-(\pi)$ is Hausdorff;
\item $\Phi^-(\mathrm{SInd}^{M_{n+1}}_{M_nV_n}(\pi/\kappa')) \cong \Phi^-(\pi)$, 
where the isomorphism is induced from map $a\circ \Theta^{-1}$ in Section \ref{s descrip bz layer}.
\end{enumerate}
\end{lemma}

\begin{proof}
The induced map $a \circ \Theta^{-1}$ from $\mathrm{SInd}_{M_nV_n}^{M_{n+1}}(\pi|_{M_n}\boxtimes \phi)$ to $\pi$ descends to a map from $\mathrm{SInd}_{M_nV_n}^{M_{n+1}}(\pi/\kappa')$ to $\pi$ by Lemma \ref{lem action map kernel}. Let $\pi'$ be the image of the map in $\pi$. Then the induced map from $a\circ \Theta^{-1}$
\[  \Phi^-(\mathrm{SInd}_{M_nV_n}^{M_{n+1}}(\pi/\kappa')) \rightarrow \Phi^-(\pi') 
\]
is still surjective.

Moreover, for $x \in \pi/\kappa'$, we construct a map $f$ in $\mathrm{SInd}_{M_nV_n}^{M_{n+1}}(\pi/\kappa')$ such that:
\[   f(m)=\widetilde{f}(m)\cdot (m.x) 
\]
for some $\widetilde{f} \in \mathrm{SInd}_{M_nV_n}^{M_{n+1}}(\mathbb C_{\mathrm{triv}}\boxtimes \phi)$. By Lemma \ref{lem phi minus map}, any element in $\Phi^-(\mathrm{SInd}_{M_nV_n}^{M_{n+1}}(\pi/\kappa'))$ can be represented by such $f$. Then the action map $a$ gives that $a \circ \Theta^{-1}(f)$ is a scalar multiple of $v$. In particular, the induced maps
\begin{align} \label{eqn composition maps}
\Phi^-(\mathrm{SInd}_{M_nV_n}^{M_{n+1}}(\pi/\kappa')) \cong \pi/\kappa' \rightarrow \Phi^-(\pi') \rightarrow \Phi^-(\pi)
\end{align}
forces that 
\[  \kappa' \subset  \mathfrak v.( \pi'\otimes \phi^{-1}) \subset \mathfrak v.(\pi \otimes \phi^{-1}).
\]

Combining with Lemma \ref{lem action map kernel}, the above inclusions are equalities, i.e.
\[  \kappa' =\mathfrak v.(\pi \otimes \phi^{-1}).
\]
This proves (1) as $\kappa'$ itself is closed. Note that (2) then follows from (1).

For (3), again the above composition (\ref{eqn composition maps}) of maps has to be a bijection from (1). This then implies (3).
\end{proof}

\subsection{Canonical embeddings}

\begin{lemma} \label{lem hausdorff quotient}
    Let $\pi$ be in $\mathrm{Rep}^{\infty, F}(M_{n+1})$. Then there exists a Hausdorff $M_{n}$-quotient $\omega$ of $\pi|_{M_n}$ such that the map $a\circ \Theta^{-1}$ in Section \ref{s descrip bz layer} induces a continuous $M_{n+1}$-equivariant bijection of from $\Phi^+(\omega)$ to $\mathcal S(\mathfrak v_n^*\setminus \left\{ 0 \right\}). \pi$.
\end{lemma}

\begin{proof}
From Section \ref{s descrip bz layer}, we have a continuous surjective map $a\circ \Theta^{-1}: \Phi^+(\pi|_{M_n}) \rightarrow \mathcal S(\mathfrak v_n^*\setminus \left\{ 0 \right\}).\pi$. The lemma now follows from Corollary \ref{cor alt form imprimitivity}. \end{proof}


\begin{corollary} \label{cor hausdorff space} (cf. \cite[Theorem 2.5.8]{dCl91})
Let $\pi \in \mathrm{Rep}^{\infty, F}(M_{n+1})$. Then there exists a continuous $M_{n+1}$-equivariant bijection from $\Phi^+\circ \Phi^-(\pi)$ to $\mathcal S(\mathfrak v_n^*\setminus \left\{ 0 \right\}).\pi$.
\end{corollary}

\begin{proof}
This follows from Lemmas \ref{lem kernel schwartz}(1) and \ref{lem hausdorff quotient}.
\end{proof}

\section{Bernstein-Zelevinsky filtrations} \label{s real bz}

\subsection{Bernstein-Zelevinsky filtrations}

\begin{definition} \label{def bz filtration}
Let $\pi$ be in $\mathrm{Rep}^{\infty, F}(M_{n+1})$. 
\begin{enumerate}
\item The induced map from $\Phi^+\circ \Phi^-(\pi)$ to $\mathcal S(\mathfrak v_n^*\setminus \left\{ 0 \right\}).\pi$ in Corollary \ref{cor hausdorff space} is called the {\it canonical embedding} from $\Phi^+\circ \Phi^-(\pi)$ to $\pi$.
\item By Lemma \ref{lem kernel schwartz} several times, for any $k=,1\ldots, n$, $(\Phi^-)^k(\pi)$ is Hausdorff and so is in $\mathrm{Rep}^{\infty, F}(M_{n+1-k})$. Thus we also have the canonical embedding:
\[  \Phi^+\circ \Phi^- ((\Phi^-)^k(\pi)) \hookrightarrow (\Phi^-)^k(\pi) .
\]
However, inductively, we also have the canonical embedding 
\[ \Phi^+\circ (\Phi^-)^k(\pi)= \Phi^+\circ \Phi^-((\Phi^-)^{k-1}(\pi)) \hookrightarrow (\Phi^-)^{k-1}(\pi).  \] 
Repeating the process, we then obtain a sequence of embeddings in $\mathrm{Rep}^{\infty, F}(M_{n+1})$:
\[   (\Phi^+)^k\circ (\Phi^-)^k(\pi) \hookrightarrow \ldots \hookrightarrow \Phi^+\circ \Phi^-(\pi) \hookrightarrow \pi .
\]
The composition of such the embeddings is called the {\it $k$-th canonical embedding}, or simply the {\it canonical embedding} again.
\item For $k=0$, define $\Lambda_0(\pi)=\pi$ and for $k = 1,\ldots, n$, define $\Lambda_k(\pi)$ as the closure of the image of the $k$-th canonical embedding.
\item We shall call
\[   0=\Lambda_{n+1}(\pi) \subset  \Lambda_n(\pi) \subset \ldots \subset  \Lambda_1(\pi) \subset \Lambda_0(\pi)=\pi 
\]
to be the (real) {\it Bernstein-Zelevinsky filtration} for $\pi$.
\end{enumerate}
\end{definition}

\begin{example}
Examples are those in Section \ref{ss   guiding example}. Let $N$ be the unipotent radical of $B_2$. In such cases, $G=\mathrm{GL}_2(\mathbb R)$ and $n+1=2$, and so 
\[  \Lambda_2(\pi)=0, \quad \Lambda_1(\pi)\cong \mathbf{CJ}^{\infty}_{s,N}(\pi), \quad \Lambda_0(\pi)=\pi, \quad \pi/\Lambda_1(\pi)\cong \mathbf{CJ}^{\infty}_N(\pi) .
\]
\end{example}

\subsection{Adjointness of BZ-functors} \label{s adjointness bz functors}

We first have the following adjointness of $\Phi^+$ and $\Phi^-$ (cf. \cite[Section 5.11]{BZ76}):

\begin{proposition} \label{prop adjoint bz functor}
\begin{enumerate}
\item $\Phi^+$ is left adjoint to $\Phi^-$ i.e. for $\pi_1$ in $\mathrm{Rep}^{\infty, F}(M_n)$ and $\pi_2$ in $\mathrm{Rep}^{\infty, F}(M_{n+1})$, 
\[  \mathrm{Hom}_{M_n}(\pi_1, \Phi^-(\pi_2))\cong \mathrm{Hom}_{M_{n+1}}(\Phi^+(\pi_1), \pi_2) .
\]
\item Let $\pi \in \mathrm{Rep}^{\infty, F}(M_{n+1})$. Under the adjointness $\mathrm{Hom}_{M_n}(\Phi^-(\pi), \Phi^-(\pi)) \cong \mathrm{Hom}_{M_{n+1}}(\Phi^+\circ \Phi^-(\pi), \pi)$, the identity map in $\mathrm{Hom}_{M_n}(\Phi^-(\pi), \Phi^-(\pi))$ corresponds to the canonical embedding from $\Phi^+\circ \Phi^-(\pi)$ to $\pi$.
\item $\Phi^-\circ \Phi^+ \cong \mathrm{Id}$.
\item Let $\pi \in \mathrm{Rep}^{\infty, F}(M_n)$. Then $\Phi^+(\pi)$ is $\mathfrak v_{n}$-stable, and so $\mathbf{CJ}^{\infty}_{V_n}\circ \Phi^+(\pi)=0$.
\item The functors $\Phi^-$ and $\Phi^+$ send a strict short exact sequence to a strict short exact sequence.
\end{enumerate}
\end{proposition}

\begin{proof}
(3) follows from Lemma \ref{lem kernel schwartz}. For (4), one may proceed with a similar argument in the proof of Lemma \ref{lem phi minus map}. One has the identification $\Phi^+(\omega)$ with $\mathcal S(\mathfrak v^*\setminus \left\{ 0\right\})\widehat{\otimes}\omega$. Now it follows from Lemma \ref{lem n-invariant under schwartz} that $\mathfrak v.\mathcal S(\mathfrak v^*\setminus \left\{ 0\right\})\widehat{\otimes}\omega=\mathcal S(\mathfrak v^*\setminus \left\{ 0 \right\})\widehat{\otimes}\omega$. Thus, $\mathfrak v.\Phi^+(\pi)=\Phi^+(\pi)$ and so this gives (4).

The $\Phi^+$-part in (5) follows from the definition of (5) and the exactness of the Schwartz induction. We now prove the exactness of the $\Phi^-$-part in (5). A contravariant functor is right exact and so it remains to show the injectivity of $\Phi^-$. Suppose we have an injection $0 \rightarrow \pi_1 \rightarrow \pi_2$. We have the following commutative diagram:
\[ \xymatrix{  & \Phi^+\circ  \Phi^-(\pi_1) \ar@{^{(}->}[d] \ar[r] &  \Phi^+\circ  \Phi^-(\pi_2)  \ar@{^{(}->}[d] \\
     0 \ar[r] &      \pi_1 \ar[r] & \pi_2
}
\]
As the above diagram commutes, we must have that the upper row is also injective, but then $\Phi^-(\pi_1)$ has to inject to $\Phi^-(\pi_2)$ as well by the exactness of $\Phi^+$.

We now consider (1). Let $f \in \mathrm{Hom}_{M_{n}}(\pi_1, \Phi^-(\pi_2))$. Recall that we have the canonical embedding $\Phi^+\circ \Phi^-(\pi_2) \hookrightarrow \pi_2$. Composing with the map 
\[ \Phi^+(\pi_1) \stackrel{\Phi^+(f)}{\rightarrow} \Phi^+\circ \Phi^-(\pi_2),\] we obtain a map from $\Phi^+(\pi_1)$ to $\pi_2$, defining a map in $\mathrm{Hom}_{M_{n+1}}(\Phi^+(\pi_1), \pi_2)$. We denote such the assignment by $F:\mathrm{Hom}_{M_n}(\pi_1, \Phi^-(\pi_2))\rightarrow \mathrm{Hom}_{M_{n+1}}(\Phi^+(\pi_1), \pi_2)$.

On the other hand, for $h \in \mathrm{Hom}_{M_{n+1}}(\Phi^+(\pi_1), \pi_2)$, we then obtain a map by applying the functor $\Phi^-$:
\[   \Phi^-\circ \Phi^+(\pi_1) \rightarrow \Phi^-(\pi_2) .
\]
From (3), this defines a map from $\pi_1$ to $\Phi^-(\pi_2)$. Then, we also obtain a map, denoted by $G:\mathrm{Hom}_{M_{n+1}}(\Phi^+(\pi_1), \pi_2) \rightarrow \mathrm{Hom}_{M_n}(\pi_1, \Phi^-(\pi_2))$. 

Now, it is straightforward to check that $G\circ F=\mathrm{Id}$ by using (3) and Lemma \ref{lem kernel schwartz}(2).  To show $F\circ G=\mathrm{Id}$, we have the following commutative diagram: for $f \in \mathrm{Hom}_{M_{n+1}}(\Phi^+(\pi_1), \pi_2)$,
\[ \xymatrix{  & \Phi^+\circ  \Phi^-(\Phi^+(\pi_1)) \ar@{^{(}->}[d] \ar[r] &  \Phi^+\circ  \Phi^-(\pi_2)  \ar@{^{(}->}[d] \\
     &      \Phi^+(\pi_1) \ar[r]^f & \pi_2 .
}
\]
But the left vertical map is a natural isomorphism and so one obtains from the diagram that $F\circ G(f)=f$, as desired. 

(2) follows from the description of the adjointness above.
\end{proof}


\subsection{Properties of Bernstein-Zelevinsky filtrations}



\begin{theorem} \label{thm bz filtration full}
Let $\pi$ be in $\mathrm{Rep}^{\infty, F}(M_{n+1})$. For $k=0,\ldots, n$, let $\pi_k=\Lambda_k(\pi)$, and let $\tau_k=\Lambda_1((\Phi^-)^{k}(\pi))$ i.e. the closure of the canonical embedding from $\Phi^+\circ \Phi^-\circ (\Phi^-)^{k}(\pi)$ to $(\Phi^-)^{k}(\pi)$. The Bernstein-Zelevinsky filtration of $\pi$ satisfies the following properties: for $k=0,\ldots, n$,
\begin{enumerate}
    \item $(\Phi^-)^k(\pi)=(\Phi^-)^k(\pi_k)$;
        \item $(\Phi^-)^k(\pi)$ is Hausdorff;
    \item $\tau_{k}$ is the maximal $\mathfrak v_{n-k}$-stable closed submodule in $(\Phi^-)^{k}(\pi_{k})$;
    \item $(\Phi^-)^{k}(\pi)/\tau_{k}$ is isomorphic to $\mathbf{CJ}_{V_{n-k}}^{\infty}((\Phi^-)^{k}(\pi))$ in $\mathrm{Rep}^{\infty,F}(M_{n-k+1})$; 
    \item there exists a continuous epimorphism from $(\Phi^+)^k\circ \mathbf{CJ}_{V_{n-k}}^{\infty}\circ (\Phi^-)^{k}(\pi))$ to $\pi_{k}/\pi_{k+1}$.
\end{enumerate}
\end{theorem}

\begin{proof}
 (2) follows from Lemma \ref{lem kernel schwartz}. (3) follows from Lemma \ref{lem hausdorff quotient}, and Corollaries \ref{cor hausdorff space} and \ref{cor cj sub max stable}. (4) follows from (3) and Proposition \ref{prop surj cj quotient}. (5) follows from the continuous epimorphism from $(\Phi^+)^k((\Phi^-)^{k}(\pi)/\tau_{k})$ to $\pi_k/\pi_{k+1}$, and (4).

 For (1), we have the following continuous embeddings:
\[  (\Phi^+)^k \circ (\Phi^-)^k(\pi) \hookrightarrow \Lambda_k(\pi) \hookrightarrow \pi .
\]
It follows from Proposition \ref{prop adjoint bz functor}(3)  (multiple times) that in the composition
\[  (\Phi^-)^k\circ (\Phi^+)^k \circ (\Phi^-)^k(\pi) \hookrightarrow (\Phi^-)^k(\Lambda_k(\pi)) \hookrightarrow   (\Phi^-)^k(\pi) ,
\]
the first term is naturally isomorphic to $(\Phi^-)^k(\pi)$. This forces $(\Phi^-)^k(\Lambda_k(\pi))\cong (\Phi^-)^k(\pi)$.
\end{proof}


\subsection{Nuclearity conjecture} \label{ss continuous map homeomorphism}






In general, the image of the canonical embedding may not be closed.

\begin{example} \label{ex small seminorm family2}
We consider the topology on $\pi_1$ as the one defined by the seminorms $p_{r,s}$  in Example \ref{ex small seminorm family}. In this case, $\mathcal S(\mathfrak v_1^*\setminus \left\{ 0 \right\}).\pi_1=\mathcal S(\mathbb R^{\times}) \subsetneq \mathbf{CJ}_s^{\infty}(\pi_1)$, as we have seen that $\mathcal S(\mathbb R^{\times})$ is not complete in $\pi_1$.
\end{example}

We state the following nuclearity conjecture, that guarantees the closedness of the canonical image:

\begin{conjecture}
Let $\pi$ be in $\mathrm{Rep}^{\infty, F}(M_{n+1})$. Suppose $\pi$ is nuclear. Then, for any $k=1,\ldots, n$,
\[   (\Phi^+)^k\circ (\Phi^-)^k(\pi) \cong \Lambda_k(\pi) .
\]
\end{conjecture}

The conjecture is motivated by that the nuclearity of a Fr\'echet space has stronger control on the topology. One expects that the topological vector space in Example \ref{ex small seminorm family} is not nuclear because there are no "enough" seminorms; for more non-nuclear examples whose image of the canonical embedding is not closed, see~\cite[Section 3.6]{WY26+}. It is known that Casselman-Wallach representations are nuclear \cite{CHM00, BK14}, and we comment on this situation based on a study in \cite{WZ25+}:

\begin{remark} 
Let $\mathrm{Ind}_{B_n}^{\mathrm{GL}_n(\mathbb K)} \chi$ be the normalized parabolically induced representation from a character $\chi$ of $B_n$. We identify $T=\mathbb K^{\times}\times \ldots \times \mathbb K^{\times}$. We write
\[  \chi =\chi_1 \boxtimes \ldots \boxtimes \chi_n .
\]
Let $\chi'=\chi_1\boxtimes \ldots \boxtimes \chi_{n-1}$. Let $\pi'=\mathrm{Ind}_{B_{n-1}}^{\mathrm{GL}_{n-1}(\mathbb K)}\chi'$. Then, by inductive argument, Mackey theory \cite{CSu21} and Fourier transform on nilradicals; roughly speaking, $\pi'$ admits a filtration whose successive subquotients are isomorphic to 
\[     (\Phi^+)^k\circ \Psi^+(\omega) ,
\]
where $\omega$ is a $\mathrm{GL}_{n-k}$-representation and $\Psi^+(\omega)$ inflates $\omega$ to a $M_{n-k+1}$-representation. Here we use "roughly speaking" as one has to consider an infinite filtration with an inverse limit such that each successive subquotients still admit an infinite filtration. The key point is that the successive subquotients provide good topologies and so one deduces that the image of the canonical embedding in Theorem \ref{thm bz filtration full} is closed. The general Casselman-Wallach case follows from the celebrated Casselman embedding trick along with the principal series case and the du Cloux imprimitivity theorem; for precise theorem, see~\cite[Theorem 1.7]{WY26+}.
\end{remark}

 \end{document}